\documentclass[a4paper,11pt,reqno]{amsart}

\usepackage{amsmath}
\usepackage{amssymb}
\usepackage{xspace,epsfig}
\usepackage{mathptmx}

\usepackage{xcolor}
\usepackage{amsmath,amsfonts,amssymb,mathrsfs,amsthm}

\numberwithin{equation}{section}

\newcommand{\R}{\mathbb R}
\newcommand{\C}{\mathbb C}
\newcommand{\N}{\mathbb N}

\newcommand{\be}{\begin{equation}}
\newcommand{\ee}{\end{equation}}
\newcommand{\ba}{\begin{eqnarray}}
\newcommand{\ea}{\end{eqnarray}}

\newcommand{\p}{\partial}

\def\e{{\varepsilon}}

\newcommand\mub{\widetilde{\mu}}

\newcommand\bna{\begin{eqnarray*}}
\newcommand\ena{\end{eqnarray*}}

\newcommand\bnan{\begin{eqnarray}}
\newcommand\enan{\end{eqnarray}}

\newcommand\bnp{\begin{proof}}
\newcommand\enp{\end{proof}}

\newcommand\bneq{\begin{eqnarray*}\left\lbrace \begin{array}{rcl}}
\newcommand\eneq{\end{array} \right.\end{eqnarray*}}
\newcommand\bneqn{\begin{eqnarray}\left\lbrace \begin{array}{rcl}}
\newcommand\eneqn{\end{array} \right.\end{eqnarray}}

\newcommand\bni{\begin{itemize}}
\newcommand\eni{\end{itemize}}

\newcommand\nor[2]{\left\|#1\right\|_{#2}}

\usepackage{hyperref}

\newtheorem{theorem}{Theorem}[section]
\newtheorem{proposition}[theorem]{Proposition}
\newtheorem{lemma}[theorem]{Lemma}
\newtheorem{corollary}[theorem]{Corollary}

\theoremstyle{remark}

\newtheorem{definition}{Definition}

\numberwithin{equation}{section}

\begin{document}
\title[Exact controllability of nonlinear heat equations]
{Exact controllability of nonlinear heat equations\\ in spaces of analytic functions}

\author[Laurent]{Camille Laurent}
\address{CNRS, Sorbonne Universit\'e, Laboratoire Jacques-Louis Lions, UMR 7598, Boite courrier 187, 75252, Paris Cedex 05, France}
\email{camille.laurent@upmc.fr}

\author[Rosier]{Lionel Rosier}
\address{Centre Automatique et Syst\`emes (CAS) and Centre de Robotique, MINES ParisTech, PSL Research University, 60 Boulevard Saint-Michel,
 75272 Paris Cedex 06, France}
\email{lionel.rosier@mines-paristech.fr}

\date{}

\maketitle

\begin{abstract}
It is by now well known that the use of Carleman estimates allows to establish the controllability to trajectories of nonlinear parabolic equations. However, by this approach, it is not clear how to decide whether a given function is indeed reachable. In this paper, we pursue the study of the reachable states
of parabolic equations based on a direct approach using control inputs in Gevrey spaces by considering a nonlinear heat equation in dimension one. The nonlinear part  is assumed to be an analytic function of the spatial variable
$x$, the unknown $y$, and its derivative $\partial _x y$.  
By investigating carefully a nonlinear Cauchy problem in the spatial variable and the relationship between the jet of space derivatives and the jet of time derivatives, 
we derive an exact controllability result for small  initial and final data that can be extended as analytic functions on some ball of the complex plane.

\end{abstract}

\vspace{0.3cm}

\textbf{2010 Mathematics Subject Classification: 35K40, 93B05}  

\vspace{0.5cm}

\textbf{Keywords:}  Nonlinear heat equation; viscous Burgers' equation; Allen-Cahn equation; exact controllability; ill-posed problems; Gevrey class; 
reachable states. 


\section{Introduction}

The null controllability of nonlinear parabolic equations is well understood since the nineties. It was 
derived in \cite{GL} in dimension one by solving some ``ill-posed problem'' with Cauchy data in some
Gevrey spaces, and in \cite{imanuvilov,FI} in any dimension and for any control region by using some
``parabolic Carleman estimates''.

The null controllability was actually extended  to the controllability to trajectories in \cite{FI}. However, 
it is a quite hard task to decide whether a given state is the value at some time of a trajectory of the system without control (free evolution). In practice, the only known examples of such states are the steady states. 

As noticed in \cite{MRRreachable}, in the linear case, the steady states are Gevrey functions of order $1/2$ in $x$
(and thus analytic over $\C$) for which infinitely many traces vanish at the boundary, a fact which is also a very conservative 
condition leading to exclude e.g. all the nontrivial polynomial functions.

The recent paper \cite{MRRreachable} used the flatness approach and a Borel  theorem to provide
an explicit set of  reachable states composed of states that can be extended as analytic functions on a ball $B$. 
It was also noticed in \cite{MRRreachable} that any reachable state could be extended as an analytic function on 
a square included in the ball $B$. We refer the reader to \cite{DE,HKT} for new sets of 
reachable states for the linear 1D heat equation, with control inputs chosen in $L^2(0,T)$.           
We notice that the flatness approach applied to the control of PDEs, first developed in \cite{LMR,DPRM,meurer,SMJ}, was revisited recently to 
recover the null controllability of (i) the heat equation in cylinders \cite{MRRheat}; (ii) a family of parabolic equations with unsmooth coefficients 
\cite{MRRparabolic}; (iii) the Schr\"odinger equation \cite{MRRschrodinger}; (iv) the Korteweg-de Vries equation with a 
control at the left endpoint \cite{MRRRkdv}.  One of the main features of the flatness approach is that it provides control inputs developed
as explicit series, which leads to very efficient numerical schemes.   

The aim of the present paper is to extend the results of \cite{MRRreachable} to nonlinear parabolic equations.  Roughly, we shall prove that
a reachable state for the linear heat equation is also reachable for the nonlinear one, provided that its magnitude is not too large and 
its poles and those of the nonlinear term 
are sufficiently far from the origin. The method of proof is inspired by \cite{GL} where a Cauchy problem in the variable $x$ is investigated. The main novelty is that we prove an {\em exact} controllability result (and not only a {\em null} controllability result as in \cite{GL}), and
we need to investigate the influence of the nonlinear terms on the jets of the time derivatives of two 
traces at $x=0$.  
Here, we do not use some series expansions of the control inputs as in the flatness approach, but we still use some Borel theorem as in \cite{petzsche,MRRreachable}.  It is unclear whether the same results could be obtained by the classical approach using the exact controllability of the
linearized system and a fixed-point argument.   

To be more precise, we are concerned with the exact controllability of the following nonlinear heat equation 
\ba
\partial _t y = \partial _x ^2 \, y + f(x , y ,  \partial _x y),&& x\in [-1,1],\ t\in [0,T],  \label{W1}\\ 
y(-1,t)=h_{-1}(t),&& t\in [0,T], \label{W2} \\
y(1,t)=h_1(t),&& t\in [0,T], \label{W3} \\
y(x,0)=y^0(x),&& x\in [-1,1], \label{W4}
\ea
where $f:\R ^3\to \R$ is analytic with respect to all its arguments in a neighborhood of $(0,0,0)$. More precisely, we assume that 
\be
f(x,0,0)=0\quad \forall x\in (-4,4), \label{AB1}
\ee 
and that 
\be
f(x,y_0,y_1)= \sum_{ (p,q,r)\in \N^3} a_{p,q,r} (y_0)^p (y_1)^q x^r  \quad \forall (x,y_0,y_1)\in (-4,4)^3, \label{AB2}
\ee
with
\be
|a_{p,q,r} | \le \frac{M}{ b_0^p b_1^q b_2^r}\qquad  \forall p,q,r\in \N \label{AB3}
\ee
for some constants
\be
M>0 , \quad b_0>4,\  b_1>4, \textrm{ and }  b_2>4. \label{AB4}
\ee
Note that $a_{0,0,r}=0$ for all $r \in \N$ by \eqref{AB1}. For $p,q\in \N$ let  
\[
A_{p,q}(x)=\sum_{r\in \N} a_{p,q,r} x^r, \quad |x|<b_2. 
\]
We infer from \eqref{AB2} and \eqref{AB3} that
\begin{eqnarray*}
f(x,y_0,y_1) &=& \sum_{\tiny\begin{array}{c}p,q\in \N ,\\ p+q>0\end{array}} A_{p,q}(x)(y_0)^p(y_1)^q, \\
|A_{p,q}(x)| &\le& \frac{M}{b_0^pb_1^q} \, \frac{1}{1-\frac{|x|}{b_2}}, \quad |x|<b_2. 
\end{eqnarray*}  
Among the many physically relevant instances of \eqref{W1} satisfying \eqref{AB1}-\eqref{AB4}, we quote:
\begin{enumerate}
\item the {\em heat equation with an analytic potential}:
\[
\partial _t y =\partial _x ^2 y + \varphi (x)y
\]
where $\varphi (x)=\sum_{r\ge 0}a_rx^r$, with $|a_r|\le M/b_2^r$ for all $r\in \N$ and some constants $M>0$, $|b_2|>4$. 
\item the {\em Allen-Cahn equation} 
\[\partial _t y=\partial _x^2 \, y +y-y^3 \] 
\item the {\em viscous Burgers' equation}
\be
\partial _t y=\partial _x^2 \, y  -y\partial  _x y.
\label{ABCD}
\ee
Note that our controllability result is still valid when the nonlinear term $-y\partial _x y$ in \eqref{ABCD} is replaced by a term like
$\varphi (x) y^p (\partial _x y)^q$ with $\varphi $ as in (1), and $p,q\in \N$.   
\end{enumerate}
Because of the smoothing effect, the exact controllability result has to be stated in a space of analytic functions (see \cite{MRRreachable} for the linear heat equation). For given $R>1$ and $C>0$, we denote by ${\mathcal R} _{R,C}$ the set 
\[
{\mathcal R}_{R,C} := \{ y:[-1,1]\to \R; \ \exists (\alpha _n)_{n\ge 0} \in \R ^\N, \ |\alpha _n| \le C \frac{n!}{R^n}\ \ \forall n\ge 0 \textrm{ and  }  
y(x) =\sum_{n=0}^\infty \alpha _n \frac{x^n}{n!} \ \  \forall x\in [-1,1]\}.   
\]
We say that a function $h\in C^\infty ([t_1,t_2])$ is {\em Gevrey of order $s\ge 0$ on $[t_1,t_2]$}, and we write  
$h\in G^s([t_1,t_2])$, if there exist
some positive constants $M,R$ such that 
\[
|\partial _t ^p h(t) | \le M \frac{(p!)^s}{R^p}\quad \forall t\in [t_1,t_2], \ \forall p\ge 0. 
\]
Similarly, we say that a function $y\in C^\infty ([x_1,x_2]\times [t_1,t_2])$ is 
{\em Gevrey of order $s_1$ in $x$ and $s_2$ in $t$}, with $s_1,s_2\ge 0$, and we write $y\in G^{s_1,s_2}([x_1,x_2]\times [t_1,t_2])$, if there exist some positive constants $M,R_1,R_2$ such that 
\[
\vert \partial _x ^{p_1}\partial _t ^{p_2} y(x,t) \vert  \le M \frac{ (p_1!)^{s_1} (p_2!)^{s_2}}{R_1^{p_1} R_2^{p_2}}\quad 
\forall (x,t)\in [x_1,x_2]\times [t_1,t_2], \ \forall (p_1,p_2)\in \N ^2.
\]

The main result in this paper is the following { \em exact controllability} result. 
\begin{theorem}
\label{thm1}
Let $f=f(x,y_0,y_1)$ be as in \eqref{AB1}-\eqref{AB4} with $b_2>\hat R := 4e^{(2e)^{-1}}\approx 4.81$. 
Let $R>\hat R$ and $T>0$. 
Then there exists some number $\hat C >0$ such that for all
 $y^0,y^1\in {\mathcal R}_{R,\hat C}$, there exists $h_{-1},h_1\in G^2([0,T])$ such that the solution $y$ of 
\eqref{W1}-\eqref{W4} is defined for all $(x,t)\in [-1,1]\times [0,T]$ and satisfies $y(x,T)=y^1(x)$ for all  $x\in [-1,1]$. Furthermore, we have that $y\in G^{1,2}([-1,1]\times [0,T])$.      
\end{theorem}

A similar result with only {\em one control} can be derived assuming that $f$ is odd w.r.t. $(x,y_0)$. 
Consider the control system
\ba
\partial _t y = \partial _x ^2 \, y + f(x , y ,  \partial _x y),&& x\in [0,1],\ t\in [0,T],  \label{WW1}\\ 
y(0,t)=0,&& t\in [0,T], \label{WW2} \\
y(1,t)=h(t),&& t\in [0,T], \label{WW3} \\
y(x,0)=y^0(x),&& x\in [0,1]. \label{WW4}
\ea
\begin{corollary}
\label{cor1}
Let $f=f(x,y_0,y_1)$ be as in \eqref{AB1}-\eqref{AB4} with $b_2>\hat R := 4e^{(2e)^{-1}}\approx 4.81$, and assume that
\be
\label{WW5}
f(-x,-y_0,y_1)= - f(x,y_0,y_1)\quad \forall x\in [-1,1], \ \forall y_0,y_1\in (-4,4). 
\ee
Let $R>\hat R$ and $T>0$. Then there exists some number $\hat C >0$ such that for all
 $y^0,y^1\in {\mathcal R}_{R,\hat C}$ with $(y^0(-x), y^1(-x))= (-y^0(x), -y^1(x))$ for all $x\in [-1,1]$, there exists $h\in G^2([0,T])$ such that the solution $y$ of 
\eqref{WW1}-\eqref{WW4} is defined for all $(x,t)\in [-1,1]\times \in [0,T]$ and satisfies $y(x,T)=y^1(x)$ for all  $x\in [0,1]$. Furthermore, we have that 
$y\in G^{1,2}([0,1]\times [0,T])$.      
\end{corollary}
 
 Corollary \ref{cor1} can be applied e.g. to (i) the heat equation with an {\em even} analytic potential;  (ii) the Allen-Cahn equation; (iii) 
 the viscous Burgers' equation.  

The constant $\hat R := 4e^{(2e)^{-1}}$ is probably not optimal, but our main aim was to provide an   
 explicit (reasonable) constant. It is expected that the optimal constant is $\hat R := 1$, with  
a diamond-shaped domain of analyticity as in \cite{DE} and \cite{HKT} for the linear heat equation.

The paper is organized as follows. Section \ref{section2} is concerned with the wellposedness of the Cauchy problem in the 
$x$-variable  (Theorem \ref{thm2}).
The relationship between the jet of space derivatives and the jet of time derivatives at some point  (jet analysis) for a solution of \eqref{W1} is studied in Section \ref{section3}. In particular, 
we show that the nonlinear heat equation \eqref{W1} can be (locally) solved forward and backward if the initial
data $y_0$ can be extended as an analytic function in some ball of $\C$ (Proposition  \ref{prop100}). 
Finally, the proofs of Theorem \ref{thm1} and Corollary \ref{cor1} are  displayed in Section \ref{section4}.  

\section{Cauchy problem in the space variable}
\label{section2}

\subsection{Statement of the global wellposedness result} 

Let $f=f(x,y_0,y_1)$ be as in \eqref{AB1}-\eqref{AB4}. 
We are concerned with the wellposedness of the Cauchy problem in the variable $x$: 
\ba
\partial _x ^2y =\partial _t y - f(x,y, \partial _x y),&&  x\in [-1,1],\  t\in [t_1,t_2], \label{C1} \\
y(0,t)=g_0(t),&& t\in [t_1,t_2], \label{C2}\\
\partial _x y (0,t)=g_1(t),&& t\in [t_1,t_2], \label{C3}     
\ea 
for some given functions $g_0,g_1\in G^2([t_1,t_2])$.
The aim of this section is to prove the following result. 
\begin{theorem}
\label{thm2}
Let $f=f(x,y_0,y_1)$ be as in \eqref{AB1}-\eqref{AB4}. 
Let $-\infty <t_1<t_2<+\infty$ and $R>4$. Then there exists some number $C>0$ such that for all 
$g_0,g_1\in G^2([t_1,t_2])$ with 
 \be
 \label{AA}
 \vert g^{ (n) } _i(t)\vert \le C\frac{ (n!)^2 }{R^{n}}, \quad i=0,1,\ n\ge 0, \ t\in[t_1,t_2],  
 \ee
there exist some numbers $R_1,R_2$ with $4/e<R_1<R_2$ and a solution $y\in G^{1,2} ([-1,1]\times [t_1,t_2])$ of \eqref{C1}-\eqref{C3} satisfying for some 
constant $M>0$
\be
\label{AAAAA}
\vert \partial _x ^{p_1}\partial _t ^{p_2} y(x,t) \vert  \le M \frac{ ( p_1 + 2 p_2) !  }{R_1^{p_1} R_2^{2p_2}}\quad 
\forall (x,t)\in [-1,1]\times [t_1,t_2], \ \forall (p_1,p_2)\in \N ^2.
\ee
\end{theorem}

\subsection{Abstract existence theorem}
We consider a family of Banach spaces $(X_s)_{s\in [0,1]}$ satisfying for $0\leq s'\leq s\leq 1$
\bnan
\label{embed}
X_s\subset X_{s'},\\
\label{CC1}
\nor{f}{X_{s'}}\leq \nor{f}{X_s};
\enan
that is, the natural embedding $ X_s\subset X_{s'}$ for $s'\le s$ is of norm less than $1$.

We are concerned with an abstract Cauchy problem  
\begin{eqnarray*}
\partial_x U(x)&=&G(x)U(x), \quad -1\le x\le 1, \\
U(0)&=&U^0,
\end{eqnarray*}
where $X^0\in X_1$ and $\big( G(x)\big)_{x\in [-1,1]}$ is a familly  of {\em nonlinear} operators with
 possible {\em loss of derivatives}.  

Our first result is a  
global wellposedness result. 
\begin{theorem}
\label{thmexistenceloc}
For any $\e \in (0,1/4)$, there exists a constant $D>0$ such that for any family $(G(x))_{x\in [-1,1]} $ of nonlinear maps 
from $X_s$ to $X_{s'}$ for $0\le s'<s\le 1$ satisfying 
\bnan
\nor{G(x)U}{X_{s'}}& \le &\frac{\e }{s-s'}\nor{U}{X_{s}}
\label{hypoGloc1}
\\
\label{hypoGloc2}
\nor{G(x)U- G(x)V)}{X_{s'}}& \le &\frac{\e }{s-s'}\nor{U-V}{X_{s}}
\enan
for $0\leq s'<s\leq 1$, $x\in [-1,1]$ and $U$,$V \in X_s$ 
with $\nor{U}{X_{s}}\le D$, $\nor{V}{X_{s}}\leq D$,  there exists $\eta>0$ so that for any $U^0\in X_1$ with $\nor{U^0}{X_1}\leq \eta$, there exists a 
solution $U\in C([-1,1],X_{s_0})$ for some $s_0\in (0,1)$ to the integral equation
\bna
U(x)=U^0+\int_0^x G(\tau)U(\tau)d\tau .
\ena
Moreover, we have the estimate 
\[
\nor{U( x )}{X_s}\leq C_1 \left(1 -\frac{ \lambda |x|}{a_\infty (1-s)} \right) ^{-1} \nor{U^0}{X_1}, \quad 
\textrm{\rm for }   \ 0\le s<1, \ |x|<\frac{a_\infty}{\lambda} (1-s),
\]
where $\lambda \in (0,1)$, $a_\infty \in (\lambda ,1)$ and $C_1>0$ are some constants.
In particular, we have 
\[
\nor{U( x )}{X_s}\leq C_1 \left(1 -\frac{ 2}{\frac{a_\infty}{\lambda} +1} \right) ^{-1} \nor{U^0}{X_1}, \quad 
\textrm{\rm for }   \ 0\le s\le s_0=\frac{1}{2}(1-\frac{\lambda}{a_\infty}), \ |x| \le 1.
\]
If, in addition,  we assume that
\be
\label{hypoGloc0}
\textrm{ for all }  U_0\in X_s \textrm{ with } \Vert U_0\Vert _{X_s} \le D,   \textrm{ the map } \tau \in [-1,1]\to G(\tau ) U_0\in X_{s'} \ \ \textrm{ is continuous}, 
\ee
then $U$ is solution in the classical sense of  
\bneqn
\label{eqnabstraite}
\partial_x U(x)&=&G(x)U(x), \quad -1\le x\le 1, \\
U(0)&=&U^0.
\eneqn
\end{theorem}
We prove first the existence of a solution on an interval $[-(1-\delta ), 1-\delta ]$, where $\delta \in (0,1)$. Next, we use a scaling argument
to obtain a solution for $x\in [-1,1]$. 
\begin{proposition}
\label{propexistenceloc}
For any $\e \in (0,1/4)$, any $\delta\in (0,1)$ and any $G$ and $D$ as in \eqref{hypoGloc0}-\eqref{hypoGloc2}, there exists some numbers $a_{\infty}>1-\delta$ and $\eta>0$ 
such that for any $U^0\in X_1$ with $\nor{U^0}{X_1}\leq \eta$, there exists a unique solution 
for $x\in (-a_{\infty},a_{\infty})$ to \eqref{eqnabstraite} in the space $Y_\infty$
(see below). Moreover, we have the estimate 
\[
\nor{U( x )}{X_s}\leq C_1 \left(1 -\frac{|x|}{a_\infty (1-s)} \right) ^{-1} \nor{U^0}{X_1}, \quad 
\textrm{\rm for }   \ 0\le s<1, \ |x|<a_\infty (1-s),
\]
where $C_1>0$ is a constant.
\end{proposition}
\bnp[Proof of Proposition \ref{propexistenceloc}]

We follow closely the proof of \cite{KN}, taking care of the choice of the constants and of the time of existence. 

Consider a sequence of numbers  $(a_k)_{k\ge 0} $ satisfying the following properties (the existence of such a sequence is
proved in Lemma \ref{lmak}, see  below):
\begin{enumerate}
\item[(i)] $a_0=1$;
\item[(ii)] $(a_k)_{k\ge 0}$ is a decreasing sequence converging to $a_{\infty}>1-\delta$;
\item[(iii)] $ \sum_{i=0}^\infty \frac{(4\e)^i}{1-\frac{a_{i+1}}{a_i}}<+\infty$.
\end{enumerate}
Next, we pick $\eta$ small enough so that $\eta \sum_{i=0}^\infty \frac{(4\e)^i}{1-\frac{a_{i+1}}{a_i}}< D$.

We define, for $k\in \N\cup \{ \infty\}$,  the (Banach) space 
$Y_k =\{ U \in \bigcap _{0\le s< 1} C(-a_k(1-s), a_k(1-s), X_s) ; \ \Vert U\Vert _{Y_k} <\infty\}$ 
with the norm
\be
\nor{U}{Y_{k}}:=\sup_{\substack{|x|<a_{k}(1-s)\\0\leq s<1}} \nor{U(x)}{X_s} \left(1-\frac{|x|}{a_{k}(1-s)}\right) \qquad \textrm{ if } k\in \N ,
\label{A1001}
\ee
\be
\nor{U}{Y_{\infty}}:=\sup_{\substack{|x|<a_{\infty}(1-s)\\0\leq s<1}} \nor{U(x)}{X_s} \left(1-\frac{|x|}{a_{\infty}(1-s)}\right) 
\quad \textrm{ if } k=\infty .
\ee

Note that for $|x|<a_k(1-s)$, $0\leq s<1$, we have that  $1-\frac{|x|}{a_k(1-s)}\in (0,1]$. Note also that we have $Y_k\subset Y_{k+1}$ and $\nor{U}{Y_{k+1}}\leq \nor{U}{Y_{k}}$, for the sequence  $(a_k)_{k\in \N}$ is decreasing.

We want to define a sequence $(U_k)_{k\ge 0}$ by the relations
\bna
U_0=0, \quad U_{k+1}=TU_k \quad \textrm{ for } k\in \N 
\ena where
\bna
(TU)(x)=U^0+\int_0^x G(\tau)U(\tau)d\tau .
\ena
Note that $U_1 (x)=(TU_0)(x)=U^0$ for $|x|<1$. 
Introduce 
\[
V_k:=U_{k+1}-U_k,\quad k\in \N. 
\]

We prove by induction on $k\in \N$ the following statements (that  contain the fact that the sequence $(U_k)_{k\in \N}$ is indeed well defined):
 \ba
 \label{Itemlambda}
\lambda_{k}:=\nor{V_k}{Y_k} &\leq& (4\e )^{k}  \eta  , \\
\label{inegUs}
\nor{U_{k+1}(x)}{X_s}\leq  \sum_{i=0}^k \frac{\lambda_i}{1-\frac{a_{i+1}}{a_i}} &\leq& D \ \textnormal{  for }|x| < a_{k+1} (1-s),
\ea
so that $G(x)U_{k+1}(x)$ is well defined in $X_{s'}$ for $|x|\leq a_{k+1} (1-s)$.

Let us first check that  \eqref{Itemlambda}-\eqref{inegUs} are valid 
 for $k=0$. For \eqref{Itemlambda}, we have that 
\[
\lambda _0 = \Vert  U_1-U_0\Vert _{Y_0} \le  \Vert U^0\Vert _{X_1} \le \eta. 
\]
For \eqref{inegUs}, we notice that 
\[
\Vert U_1\Vert _{X_s} =\Vert U^0\Vert _{X_s} \le \Vert U^0\Vert _{X_1} \le 
\eta  \le \frac{\eta }{1-a_1} \le D.
\]

Assume that \eqref{Itemlambda}-\eqref{inegUs} are true up to the rank $k$. Let us check that they are also true at the rank $k+1$.

Take $s$ and $x$ (for simplicity, we assume $x\geq 0$) so that $0\leq x 
< a_{k+1}(1-s)$.
For any $0\leq \tau\leq x$, \eqref{inegUs} gives $\max \big( \nor{U_{k+1}(\tau)}{s} , \nor{U_{k}(\tau)}{s} \big) \leq D$ (recall $a_{k+1}\leq a_k$).
In particular, we can apply \eqref{hypoGloc2} replacing $s'$ by $s$ and $s$ by 
$s(\tau)=\frac{1}{2}\left(1+s-\frac{\tau}{a_{k+1}}\right)$, obtaining  
\bna
\nor{G(\tau) U_{k+1} (\tau ) - G(\tau ) U_k(\tau)}{X_{s}}&\leq &\frac{\e }{s(\tau ) -s}\nor{U_{k+1} (\tau ) -U_k(\tau)}{X_{s(\tau)}}
\ena
Note that we have indeed $0\le s<s(\tau) <1$. Next 
\bna
\nor{V_{k+1}(x)}{X_s} &=& \Vert  (T U_{k+1} )(x)   - (T U_k )(x) \Vert _{X_s} \\
&\leq& \int_0^x \nor{ G(\tau) U_{k+1} (\tau ) - G(\tau ) U_k(\tau)}{ X_s}d\tau \\
&\leq& \int_0^x\frac{\e }{s(\tau)-s}\nor{U_{k+1} (\tau ) -U_k(\tau)}{X_{s(\tau)}}d\tau\\
&\leq &\e\nor{V_k}{Y_{k}}\int_0^x\frac{a_{k+1} }{s(\tau)-s}\left(\frac{1-s(\tau)}{a_{k+1}(1-s(\tau))-\tau}\right)d\tau
\ena
where we have used the fact that $s(\tau)$ satisfies 
$\tau < a_{k+1}(1-s(\tau))$  (for 
$a_{k+1} (1-s (\tau )) -\tau  =\frac{1}{2} (a_{k+1} (1-s) -\tau )>0$) 
and $0<s(\tau)<1$, so that with 
\eqref{A1001} 
\bnan
\label{justifstau}
\nor{V_k(\tau)}{X_{s(\tau)}}
\leq \left(1-\frac{|\tau|}{a_{k+1}  (1-s(\tau) )}\right)^{{-1}}\nor{V_k}{Y_{k+1}}
\leq \left(1-\frac{|\tau|}{a_{k+1}  (1-s(\tau) )}\right)^{{-1}}\nor{V_k}{Y_k}.
\enan

Let us go back to the estimate of the integral. To simplify the notations, we denote $A=a_{k+1}(1-s)$ and recall $0\leq \tau \leq x < A$. We have
\begin{eqnarray*}
\int_0^x\frac{a_{k+1}(1-s(\tau))}{(s(\tau)-s)(a_{k+1}(1-s(\tau))-\tau)}d\tau
&=&2a_{k+1}\int_0^x\frac{A+\tau}{(A-\tau)^2}d\tau\\
&\leq& \int_0^x\frac{4a_{k+1}A}{(A-\tau)^2}d\tau=a_{k+1} \left.\frac{4A }{A-\tau}\right]_0^x\leq a_{k+1}\frac{4A }{A-x}\cdot
\end{eqnarray*}
So, recalling $\frac{A }{A-x}
=\left(1-\frac{x}{A }\right)^{-1}=  \left(1-\frac{|x|}{a_{k+1}(1-s)}\right)^{-1}$, we have obtained
\bna
\nor{V_{k+1}(x)}{s}\left(1-\frac{|x|}{a_{k+1}(1-s)}\right)\leq 4a_{k+1}\e \nor{V_k}{Y_{k}} .
\ena
So, we have proved that
\bnan
\label{Vk1}
\nor{V_{k+1}}{Y_{k+1}}\leq  4a_0\e\nor{V_k}{Y_k},
\enan
and hence $\lambda_{k+1}\leq 4a_0\e \lambda_k$. This yields
 \eqref{Itemlambda} at rank $k+1$.
 
 Let us proceed with the proof of \eqref{inegUs} at rank $k+1$. 

Since $U_{k+2}=U_{k+1}+V_{k+1}$, we only need to prove $\nor{V_{k+1}(x)}{s}\leq \frac{\lambda_{k+1}}{1-\frac{a_{k+2}}{a_{k+1}}}$ for $|x| <  a_{k+2}(1-s)$. This is obtained by noticing that 
\bna
\nor{V_{k+1}(x)}{s}\leq \left(1-\frac{|x|}{a_{k+1}(1-s)}\right)^{-1}\nor{V_{k+1}}{Y_{k+1}}\leq \left(1-\frac{a_{k+2}}{a_{k+1}}\right)^{-1}\lambda_{k+1}
\ena
since $|x| < a_{k+2}(1-s)$.
The proof by induction of  \eqref{Itemlambda}-\eqref{inegUs} is complete. 

We are now in a position to prove the existence of a solution to \eqref{eqnabstraite}.
Let us introduce the function $U_{\infty}:=\lim _{k\to +\infty} U_{k}=\sum_{k=0}^{+\infty} V_k$. Note that the convergence 
of the series is normal in $Y_{\infty}$. Indeed, 
\bna
\sum_{k=0}^{+\infty} \nor{V_k}{Y_{\infty}}\leq \sum_{k=0}^{+\infty}  \nor{V_k}{Y_{k}}\leq \sum_{k=0}^{+\infty} \lambda_k <+\infty.
\ena 
Note also that \eqref{inegUs} remains true for $U_{\infty}$ for $|x| < a_{\infty} (1-s)$, so that $G(\tau)U_{\infty}$ is well defined.

Let us prove that $U_{\infty}$ is indeed a solution of \eqref{eqnabstraite}.
Using the fact that $U_{k+1}-TU_k=0$, we have 
\bna
U_{\infty}-(TU_{\infty})(x)=U_{\infty}-U_{k+1}+\int_0^x \left[G(\tau)U_k(\tau)-G(\tau)U_{\infty}(\tau)\right]d\tau
\ena
where all the terms in the equation are in $Y_{\infty}$. 
The same estimates as before yield for $0\le s<1$ and $|x|<a_\infty (1-s)$
\[
\Vert \int_0^x [G(\tau ) U_k(\tau ) -G(\tau ) U_\infty (\tau )]\, d\tau \Vert _{X_s}
\le 4 a_\infty \varepsilon \Vert U_k-U_\infty \Vert _{Y_\infty} 
\left(  1- \frac{ |x| }{ a_\infty (1-s)} \right) ^{-1} ,  
\]
and hence
\[
\Vert U_\infty -TU_\infty\Vert _{Y_\infty}
\le \Vert U_\infty -U_{k+1} \Vert _{Y_\infty} + 4a_\infty \varepsilon 
\Vert U_k-U_\infty \Vert _{Y_\infty} \to 0  
\]
as $k\to +\infty$. Thus $U_\infty$ is a solution of \eqref{eqnabstraite}. 

Let us prove the uniqueness of the solution of \eqref{eqnabstraite} in the same space $Y_{\infty}$. Assume that $U$ and $\widetilde U$ are two solutions of \eqref{eqnabstraite} in $Y_{\infty}$. 
Pick  $\lambda \in (0,1)$, and let $a_0'=1$, $a_k'=\lambda a_k$ for $k\in \N^*\cup \{ \infty \}$. 
Notice that (i), (ii) and (iii) are still valid for the $a_k'$. 
Denote by $Y_k'$, $k\in \N \cup \{ \infty \}$, the space associated with $a_k'$. Note that $a_k'<a_\infty$ and hence
$Y_\infty \subset Y_k'$ for $k \gg 1$.   
Then we have by the same computations as above that
\[
\Vert U-\widetilde U\Vert _{Y_k'}
\le (4\varepsilon )^k \Vert U-\widetilde U \Vert _{Y_0'} \
\]
which yields 
\[
\Vert U -\widetilde U\Vert _{Y'_\infty} \le \lim_{k\to +\infty}  
\Vert U-\widetilde U\Vert _{Y_k'} =0.
\]
Thus $U(x)=\widetilde U(x)$ for $|x|<a_\infty '=\lambda a_\infty$, with
$\lambda$ as close to $1$ as desired.  
\enp
It remains to prove the existence of the sequence $(a_k)_{k\ge 0}$. This is done in the next lemma. 
\begin{lemma}
\label{lmak}
There exists a sequence $(a_k)_{k\in \N} $ satisfying (i), (ii) and (iii). 
\end{lemma}
\bnp
We denote $C_0=4\e<1$ and we require
$ \sum_{i=0}^\infty \frac{C_0^i}{1-\frac{a_{i+1}}{a_i}}<+\infty$. Picking $a_0=1$ and $\gamma>0$ small enough, we define 
the sequence $(a_k)_{k\in \N}$ by induction by setting
\bna
a_{k+1}=a_k\left(1-\frac{\gamma}{(1+k)^2}\right), \quad k\in \N. 
\ena 
The sequence $(a_k)_{k\in \N} $ is clearly decreasing, $\frac{C_0^i}{1-\frac{a_{i+1}}{a_i}}= \gamma^{-1}  (1+i)^2 C_0^i$, and hence 
$ \sum_{i=0}^\infty \frac{C_0^i}{1-\frac{a_{i+1}}{a_i}}<+\infty$, for $C_0<1$.
Finally, $b_k=\ln(a_k)$ converges to $\sum_{k=0}^{\infty} \ln\left(1-\frac{\gamma}{(1+k)^2}\right)\geq - 2\gamma \sum_{k=0}^{\infty} \frac{1}{(1+k)^2}$ for $\gamma$ small enough. In particular, $a_{\infty}\geq e^{-2  \gamma \zeta(2)}$ which can be made greater than $1-\delta$ for $\gamma$ small.
\enp
Let us complete the proof of Theorem \ref{thmexistenceloc} by using a scaling argument.
Pick any number $\lambda\in (0,1)$ with 
\[
\frac{\varepsilon}{\lambda}  <\frac{1}{4} \cdot
\]
Let $\delta =1-\lambda \in (0,1)$ and pick $a_\infty \in (\lambda , 1)$. 
Introduce the new variables $\tilde x :=\lambda x \in [-\lambda , \lambda ]
=[-(1-\delta ) , 1-\delta ]$ for $x\in [-1,1]$,  and the new unknown
\[
\tilde U (\tilde x) := U(x)=U (\lambda ^{-1} \tilde x).
\] 
Then $\tilde U$ should solve 
\begin{eqnarray*}
&&\partial _{\tilde x} \tilde U =\lambda ^{-1} G(\lambda ^{-1} \tilde x) U(\lambda ^{-1} \tilde x ) =\tilde G(\tilde x) \tilde U (\tilde x), \\
&&\tilde U (0)=U^0,
\end{eqnarray*}
where 
\[
\tilde G (\tilde x) :=
\left\{ 
\begin{array}{ll}
\lambda ^{-1} G(\lambda ^{-1} \tilde x) \ &\textrm{ if } \tilde x\in [-\lambda , \lambda ], \\
\lambda ^{-1} G(1) &\textrm{ if } \tilde x\in [\lambda ,1], \\
\lambda ^{-1} G(-1) &\textrm{ if } \tilde x\in [-1, -\lambda ].
\end{array}
\right. 
\]
Then $ \big( \tilde G (\tilde x) \big) _{\tilde x\in [-1,1]}$ is a family of nonlinear maps 
from $X_s$ to $X_{s'}$ for $0\le s'<s\le 1$ satisfying for $0\le s'< s\le 1$, 
$\tilde x\in [-1,1]$ and $U,V\in X_s$ with $\Vert U\Vert _{X_s} \le D$, 
$\Vert V\Vert _{X_s}\le D$
\begin{eqnarray*}
\Vert \tilde G (\tilde x) U\Vert _{X_{s'}} 
&\le& \frac{\tilde \varepsilon}{s-s'} \Vert U\Vert _{X_s} \\
\Vert \tilde G (\tilde x) U -\tilde G (\tilde x) V\Vert _{X_{s'}} 
&\le& \frac{\tilde \varepsilon}{s-s'} \Vert U-V\Vert _{X_s} 
\end{eqnarray*} 
where $\tilde \varepsilon =\varepsilon/\lambda \in (0,1/4)$. 
Since \eqref{hypoGloc0}-\eqref{hypoGloc2} are satisfied, we infer from Proposition
\ref{propexistenceloc} the existence of a solution $\tilde U$ of 
\[
\partial _{\tilde x} \tilde U =\tilde G (\tilde x) \tilde U (x), \quad \tilde U (0)=U^0
\] 
for $x\in [ - (1-\delta ), 1-\delta ] = [ - \lambda , \lambda ]$, and satisfying  
\[
\tilde U \in \bigcap _{0\le s<1} C(-a_\infty (1-s), a_\infty (1-s), X_s).
\]
Then the function 
$U$ defined for $x\in [-1,1]$ by $U(x)=\tilde U (\tilde x)$ solves 
\[
\partial _{x} U = G ( x) U (x), \quad U (0)=U^0,
\] 
for $x\in [-1,1]$, and it satisfies 
$U\in \bigcap_{0\le s<1} C ( -\frac{a_\infty}{\lambda} (1-s), \frac{a_\infty}{\lambda} (1-s), X_s)$ 
and
\[
\nor{U( x )}{X_s}\leq C_1 \left(1 -\frac{ \lambda |x|}{a_\infty (1-s)} \right) ^{-1} \nor{U^0}{X_1}, \quad 
\textrm{\rm for }   \ 0\le s<1, \ |x|<\frac{a_\infty}{\lambda} (1-s).
\]
The proof of Theorem \ref{thmexistenceloc}
 is complete. \qed
\subsection{Gevrey type functional spaces}
We follow closely \cite{KY,yamanaka}.
\subsection{Definitions}
We define several spaces of Gevrey $\lambda$ functions for $\lambda>1$. For our application to the heat equation, 
we shall take $\lambda=2$, but for the moment we stay in the generality. Introduce
\bna
\Gamma_{\lambda}(k)=2^{-5}(k!)^{\lambda}(1+k)^{-2},
\ena
and let $\Gamma$ denote the Gamma function of Euler. It is increasing on $[2,+\infty )$.

We also introduce a variant of those functions with a parameter $a\in\R$ ($a$ is not necessarily an integer): 
\bnan
\label{defgammaa}
\Gamma_{\lambda,a}(k)&=&2^{-5}(\Gamma(k+1 -a))^{\lambda}(1+k)^{-2}, \quad \textrm{ for }  k\in \N \ \textrm{ s.t. } k> |a|+1,  \\
\Gamma_{\lambda,a}(k)&=&\Gamma_{\lambda}(k), \quad \textrm{ for }  k\in \N \ \textrm{ s.t. }  0\leq k\leq  |a|+1.
\enan
Clearly, $\Gamma _{\lambda ,0}=\Gamma _\lambda$. 
Note that for $k> |a|+1 $, we have $k +1 -a\geq 2$, so we are in an interval where $\Gamma$ is increasing. Thus  we have for all $k\in \N$
\bnan
\label{Gammacroissance}
\Gamma_{\lambda,a}(k)&\leq& \Gamma_{\lambda}(k),  \quad\textnormal{ if } a\geq 0\\
\label{Gammacroissance2}
\Gamma_{\lambda}(k)&\leq& \Gamma_{\lambda,a}(k),  \quad  \textnormal{ if } a\leq 0.
\enan
The intermediate space will be the set of functions in $C^{\infty}(K)$ (where $K=[t_1,t_2]$ with $-\infty <t_1<t_2<\infty$) such that
\bna
\left|u\right|_{L,a}:=\sup \left\{\frac{\left|u^{(k)}(t)\right|}{L^{|k-a|}\Gamma_{\lambda,a}(k)}, t\in K, k\in\N\right\}<\infty .
\ena
Note that for $a=0$, we recover the spaces defined earlier in \cite{yamanaka}, and $\left|u\right|_{L,0}=\left|u\right|_{L}$.

\begin{definition}
\label{defGevrey}
Yamanaka \cite{yamanaka} defined the norms
\bna 
\nor{u}{L}:=\max \left\{2^6\nor{u}{L^{\infty}(K)},2^3L^{-1}\left|u'\right|_L\right\} ,
\ena
and similarly we define for $a\in \R$
\bna
\nor{u}{L,a}:=\max \left\{2^6\nor{u}{L^{\infty}(K)},2^3L^{-1}\left|u'\right|_{L,a}\right\} .
\ena
We denote by $G_L^{\lambda}$  (resp.  $G_{L,a}^{\lambda}$) the (Banach) space of functions $u\in C^{\infty}(K)$ such 
that $\nor{u} {L} <\infty$  (resp. $\nor{u}{L,a} <\infty$). 
\end{definition}
The space  $G_{L,a}^{\lambda}$  is supposed to ``represent'' the space of functions Gevrey $\lambda$ with radius $L^{-1}$  
with $a$ derivatives. Roughly, we may think that $u\in G_{L,a}^{\lambda}$ if $D^a u\in G^{\lambda}_L$, even if it is not completely true if $a\notin \N$.

Note that, as a direct consequence of 
\eqref{Gammacroissance}-\eqref{Gammacroissance2}, we have the embeddings $G_{L,a}^{\lambda}\subset G_L^{\lambda}$ if $a\geq 0$, $G_L^{\lambda} \subset G_{L,a}^{\lambda}$ if $a\leq 0$, together with the inequalities
\bnan
\label{Gacroissance+}\nor{u}{L}\leq \max(L^{a},L^{-a})\nor{u}{L,a} \textnormal{ if } a\geq 0,\\
\label{Gacroissance-}\nor{u}{L,a}\leq \max(L^{a},L^{-a}) \nor{u}{L} \textnormal{ if } a\leq 0.
\enan

Furthermore, for any $a\in \R$ and $0<L<L'$, we have the embedding $G_{L,a}^{\lambda}\subset G_{L',a}^{\lambda}$ with
\bnan
\label{GacroissanceL}\nor{u}{L',a}\leq \nor{u}{L,a}.
\enan


The following result \cite[Theorem 5.4]{yamanaka} will be used several times in the sequel. 
\begin{lemma} \label{algebre}
\textrm{(Algebra property) \cite[Theorem 5.4]{yamanaka}}
\be
\nor{uv}{L}\leq \nor{u}{L}\nor{v}{L}\quad \forall u,v\in G_L^\lambda. 
\ee
\end{lemma}

\subsection{Cost of derivation}
The following result is a variant of Proposition 2.3 of Kawagishi-Yamanaka \cite{KY}, where the spaces we consider contain some non-integer ``derivatives".
\begin{lemma}[Cost of derivatives for Gevrey spaces containing derivatives]
\label{coutderivee}
Let $\lambda>0$ and $\delta >0$. Let $q\in \N$ and $a,b\in \R$ with $d=q-a+b>0$. Then there exists some number $C=C(\lambda,\delta,a,b,q)$ such that for all $L>0$, all $\alpha>1$, and all $u\in G_{L,a} ^\lambda$, we have 
\bna
\left|u^{(q)}\right|_{\alpha L,b}\leq \left(C\left\langle L\right\rangle^C+ (1+\delta)\alpha^b L^d \left(\frac{\lambda d}{e\ln \alpha}\right)^{\lambda d}\right) \left|u\right|_{L,a}
\ena
and hence
\bna
\nor{u^{(q)}}{\alpha L,b}\leq \left(C\left\langle L\right\rangle^C+ (1+\delta)\alpha^b L^d \left(\frac{\lambda d}{e\ln \alpha}\right)^{\lambda d}\right)\nor{u}{L,a}.
\ena
\end{lemma}
\bnp
The main tool will be the asymptotic of the Gamma function $\frac{\Gamma(x+d)}{\Gamma(x)}\sim x^d$ as $x\to +\infty$, which follows at once from Stirling's formula 
(see \cite{rudin})
\[
\lim_{x\to +\infty} \frac{\Gamma (x+1) }{  (x/e)^x \sqrt{2\pi x}} =1.
\]  
In particular, for any $\delta>0$, there exists a number $N=N(\lambda,\delta,a,b,q)$ such that for all $k\in \N$ with $k\ge N$,
\bna
\left(\frac{\Gamma(k+1+q-a)}{ \Gamma(k+1-b)}\right)^{\lambda }\leq (1+\delta) k^{\lambda d}.
\ena
We can also assume that $k\geq N$ implies $k+q > |a|+1$, and $k > |b|+1$, so that 
$\Gamma _{\lambda ,a}(k+q)$ and $\Gamma _{\lambda , b}(k)$ are given by  \eqref{defgammaa}. 
Note that we always have $\frac{(1+k)^2}{(1+k+q)^2} \leq 1$ if $k\in\N $, for $q\geq 0$.

For $k\geq N$, we have
\bna
\frac{|u^{(k+q)} (t) |}{(\alpha L)^{k-b} \Gamma_{\lambda,b}(k)}
&\leq &\frac{\left|u\right|_{L,a}L^{k+q-a} \Gamma_{\lambda,a}(k+q)}{(\alpha L)^{k-b} \Gamma_{\lambda,b}(k)} \\
&\leq& \frac{\left|u\right|_{L,a}L^{d} }{\alpha ^{k-b}}\left(\frac{\Gamma(k+1 +q-a)}{ \Gamma(k+1-b)}\right)^{\lambda}\\
&\leq&(1+\delta) \frac{\left|u\right|_{L,a}L^{d} }{\alpha^{k-b}}k^{\lambda d}\\
&\leq& (1+\delta) \alpha^b \left|u\right|_{L,a}L^{d} \sup_{t\geq 0} (\alpha^{-t}t^{\lambda d} )\\
&\leq&  (1+\delta) \alpha^b \left|u\right|_{L,a}L^{d} \left(\frac{\lambda d}{e \ln(\alpha)}\right)^{\lambda d},
\ena
where we used the fact that $\sup_{t\geq 0} (\alpha^{-t}t^{\lambda d})=\left(\frac{\lambda d}{e \ln(\alpha)}\right)^{\lambda d}$, where $\alpha>1$ and $\lambda d>0$. If $k\leq N$, we still have
\bna
\frac{|u^{(k+q)} (t) |}{(\alpha L)^{|k-b|} \Gamma_{\lambda,b}(k)}
&\leq &  \frac{\left|u\right|_{L,a}L^{|k+q-a|-|k-b|}}{\alpha ^{|k-b|}} \frac{ \Gamma_{\lambda,a}(k+q)}{ \Gamma_{\lambda,b}(k)} \cdot
\ena
This yields the result, for $|k+q-a|-|k-b|\leq C(\lambda,\delta,a,b,q) $, $\alpha ^{- | k-b |  } \le 1$, and 
$\frac{ \Gamma_{\lambda,a}(k+q)}{ \Gamma_{\lambda,b}(k)} \leq C(\lambda,\delta,a,b,q) $ for $k\leq N(\lambda,\delta,a,b,q)$.

The second statement follows by using the definition of $\Vert \cdot \Vert _{\alpha L,b}$ and the estimate
\bna
\nor{u^{(q)}}{L^{\infty}(K)}\leq L^{|q-1-a|} \Gamma_{\lambda,a}(q-1)\left|u'\right|_{L,a} \leq C(\lambda,q,b,a)\nor{u}{L,a}
\ena
 for $q\geq 1$, the case $q=0$ being immediate.
\enp

\subsection{Application to the semilinear heat equation}
We aim to solve the system: 
\ba
\partial _x ^2u =\partial _t u - f(x,u, \partial _x u),&&  x\in [-1,1],\  t\in [0,T], \label{DDD1} \\
u(0,t)=u_0(t),&& t\in [0,T], \label{DDD2}\\
\partial _x u (0,t)=u_1(t),&& t\in [0,T], \label{D3}     
\ea 
This is equivalent to solve the first order system
\bnan
\label{eqnU}
\partial_x U &=& A U +F(x,U),\\
\label{cdini}U(0)&=&U_0
\enan
with $U=(u,\partial_x u)$, $A  =\left(\begin{array}{cc}0 &1\\ \partial_t &0\end{array}\right)$, and $F(x,(u_0,u_1))=\left(\begin{array}{c}0\\ -f(x,u_0,u_1)\end{array}\right)$.

Let $L>0$. We define the space ${\mathcal X}_L:=\{ U=(u_0,u_1)
\in G_{L, \frac{1}{2}}^\lambda \times G_L^\lambda \}$, with
\bna
\nor{U}{{\mathcal X}_L}= \nor{ (u_0,u_1) }{{\mathcal X}_L}=
 \nor{u_0}{L,1/2}+\nor{u_1}{L},
\ena
where the norms are those defined in Definition \ref{defGevrey} with $\lambda=2$. 
(Note that $u_0$ is more regular than $u_1$ of one half derivative.) In particular, we have that
\bna
\nor{AU}{{\mathcal X}_{L}}=\nor{u_1}{L,1/2}+\nor{ \partial_t u_0}{L} .
\ena
In the following result,  $L_1$ stands for the inverse of the radius  of the initial datum.
\begin{theorem}
\label{thm4}
Pick any $L_1<1/4$. Then there exists a number $\eta>0$ such that for any $U_0\in {\mathcal X}_{L_1}$ 
with $\nor{U_0}{{\mathcal X}_{L_1}}\leq \eta$, there exists a 
solution to \eqref{eqnU}-\eqref{cdini}
 for $x\in [-1,1]$ in 
$C([-1,1], {\mathcal X}_{L_0})$ for some $L_0>0$. 
\end{theorem}
\noindent {\em Proof of Theorem \ref{thm4}.} 
In order to apply Theorem \ref{thmexistenceloc}, we introduce a scale of Banach spaces $(X_s)_{s\in [0,1]}$ as follows: 
for $s\in [0,1]$, we set
\ba
\nor{U}{X_s}&=&e^{-\tau (1-s)}\nor{(u_0,u_1)}{{\mathcal X}_{L(s)}} =e^{-\tau (1-s)}\left( \nor{u_0}{L(s),1/2}+\nor{u_1}{L(s)}\right), \label{PP0}\\
L(s)&=&e^{r(1-s)}L_1  \label{PP1}
\ea
where $r=2$ and $\tau>0$ will be chosen thereafter. Note that \eqref{CC1} is satisfied because of \eqref{GacroissanceL} and the fact that $L(s')>L(s)$ for $s'<s$. Actually, we have even that
\bnan
\label{gainXs}
\nor{U}{X_{s'}}\leq e^{-\tau (s-s')}\nor{U}{X_{s}}.
\enan
Lemma \ref{lmAder} and Lemma \ref{lemmeF} (see below) will allow us to select parameters so that $G=A+F$ satisfies the assumptions of Theorem \ref{thmexistenceloc}.
\begin{lemma}
\label{lmAder}
Let $L_1<1/4$. There exist $\tau_0>0$ large enough and $\e_0<1/4$ such that we have the estimates
 \bna
\nor{AU}{X_{s'}}&\leq &\frac{\e_0}{s-s'}\nor{U}{X_{s}} \quad \forall U\in X_s,
\ena
for all $\tau \ge \tau _0$ and all $s,s'$ with  $0\leq s'<s\leq 1$.
\end{lemma}
\bnp
By assumption, we have $\frac{ L_1^{1/2}}{2}<1/4$ and we can pick $\delta>0$ so that 
\bnan
\label{choix14}
(1+\delta)\frac{ L_1^{1/2}}{2}<1/4.
\enan

Applying Lemma \ref{coutderivee} with $\lambda=2$ and $\delta >0$ as in \eqref{choix14}, and with 
$q=0$, $b=1/2$, $a=0$ (respectively $q=1$, $b=0$, $a=1/2$), so that $\lambda d=1$ in both cases, we obtain the existence of some number $C=C_{\delta} >0$ such that
\bnan
\label{estimderivA}
\nonumber\nor{AU}{{\mathcal X}_{\alpha L}}&=&\nor{u_1}{\alpha L,1/2}+\nor{ \partial_t u_0}{\alpha L}\\
\nonumber&\leq &\left(C\left\langle L\right\rangle^C+\frac{1+\delta}{e\ln \alpha}(\alpha L)^{1/2} \right) \left(\nor{u_1}{L}+\nor{u_0}{L,1/2}\right)\\
&\leq &\left(C\left\langle L\right\rangle^C+\frac{1+\delta}{e\ln \alpha}(\alpha L)^{1/2} \right)\nor{U}{{\mathcal X}_{L}}
\enan
uniformly for $\alpha>1$ and $L>0$.
So, \eqref{estimderivA} applied with $L=L(s)$, $\alpha=\frac{L(s')}{L(s)}=e^{r(s-s')} >1$ becomes for $s'<s$ (we also use $L_1\leq C$, for $0<L_1<1/4$)  
\bna
\nor{AU}{X_{s'}}&\leq & e^{-\tau (s-s')} \left(Ce^{rC}+(1+\delta)\frac{e^{r\frac{1-s'}{2}}L_1^{1/2}}{e r(s-s')}\right)\nor{U}{X_{s}}\\
&\leq & \left(Ce^{-\tau (s-s')}e^{rC}+(1+\delta)e^{\frac{r}{2}}\frac{L_1^{1/2}}{er(s-s')}\right)\nor{U}{X_{s}}\\
&\leq &\left(\frac{e^{-1}}{\tau (s-s')} Ce^{rC}+(1+\delta)\frac{e^{\frac{r}{2}}L_1^{1/2}}{er(s-s')}\right)\nor{U}{X_{s}}
\ena
where we have used 
\bnan
\label{estimexp}
e^{-\tau (s-s')} =\frac{\tau (s-s')e^{-\tau (s-s')}}{\tau (s-s')}\leq \frac{e^{-1}}{\tau (s-s')}
\enan
 and the fact that $te^{-t}\le e^{-1}$ for $t\ge 0$. 
 Minimizing the constant in the r.h.s. leads to the choice $r=2$. 
(Note that the initial space $X_{1}={\mathcal X}_{L_1}$ is independent on the choice of $r$.)
We arrive to the estimate
\bna
\nor{AU}{X_{s'}}&\leq &\left(\frac{C}{\tau}  +(1+\delta)\frac{ L_1^{1/2}}{2}\right)\frac{1}{s-s'}\nor{U}{X_{s}} \cdot
\ena
By \eqref{choix14}, we can then pick $\tau _0$ large enough so that $\frac{C}{\tau}  +(1+\delta)\frac{ L_1^{1/2}}{2}<1/4$ for $\tau \ge \tau _0$. 
The proof of Lemma \ref{lmAder} is complete. 
\enp 
\begin{lemma}
\label{lemmeF}
Let $f$ be as in \eqref{AB1}-\eqref{AB4}, and let 
$F(x,U)=\left(\begin{array}{c}0\\ -f(x,u_0,u_1)\end{array}\right)$ for $x\in [-1,1]$ and $U=(u_0,u_1)\in L^\infty (K)^2$ with 
$\sup (\Vert u_0\Vert _{L^\infty (K) }, \Vert u_1\Vert _{L^\infty (K) } )<4$.
Let $r=2$, $L_1>0$, and $\e>0$.
Then there exists $\tau _0>0$ (large enough) such that for $\tau \ge \tau _0$, 
there exists $D>0$ (small enough) such that we have the estimates
\ba
\nor{F(x,U)}{X_{s'}}&\leq &\frac{\e}{s-s'}\nor{U}{X_s}  \label{R1}\\
\nor{F(x,U)-F(x,V)}{X_{s'}}&\leq& \frac{\e}{s-s'} \nor{U-V}{X_s} \label{R2}
\ea
for $0\leq s'<s\leq 1$, and $U=(u_0,u_1)\in X_s,V=(v_0,v_1)\in X_s$ with
\be
\nor{U}{X_{s}}\le D, \ \nor{V}{X_{s}}\leq D. \label{R3}
\ee
Finally, for $0\le s\le 1$ and $U=(u_0,u_1)\in X_s$ with 
$\Vert U\Vert _{X_s} \le D$, the map $x\in [-1,1]\to F(x,U)\in X_s$ is continuous. 
\end{lemma}
\bnp
Since \eqref{R1} follows from \eqref{R2}, for $F(x,0)=0$, it is sufficient to prove \eqref{R2}. Pick 
$0\le s'<s\le 1$, $D>0$ and $U,V\in X_s$ satisfying \eqref{R3}. Then 
\begin{eqnarray*}
\Vert F(x,U)-F(x,V) \Vert _{X_{s'}}
&=& 
\Vert 
-\left( 
\begin{array}{c}
0\\
f(x,U)-f(x,V)
\end{array}
\right) 
\Vert _{X_{s'}}\\
&= &  e^{-\tau (1-s') } \Vert f(x,u_0,u_1)-f(x,v_0,v_1)\Vert _{L(s')} \\
&\le& e^{-\tau (1-s')} \sum_{p+q>0} \Vert  A_{p,q}(x)  [u_0^p u_1^q -v_0^pv_1^q] \Vert  _{L(s')}\\
&\le& e^{-\tau (1-s')} \sum_{p+q>0} \vert A_{p,q}(x)\vert \big(   \Vert u_0^p-v_0^p\Vert _{L(s')} \Vert u_1 ^q\Vert _{L(s')}  \\
&& \qquad\qquad +\Vert v_0 ^p\Vert _{L(s')}\Vert u_1^q-v_1^q\Vert _{L(s')} \big)
\end{eqnarray*} 
where we used the triangle inequality and Lemma \ref{algebre}. 
Note that, by \eqref{Gacroissance+}, we have for a constant $C=C(L_1)\ge 1$ and any $0\le s'<1$
\be
\label{TTT1}
\nor{u_0}{L(s')}+\nor{u_1}{L(s')}\leq C \nor{u_0}{L(s'),1/2}+\nor{u_1}{L(s')}
\leq C e^{  \tau(1-s')}\nor{U}{X_{s'}}\leq CD e^{\tau}, 
\ee
and similarly 
\bna
\nor{v_0}{L(s')}+\nor{v_1}{L(s')}\leq CD e^{\tau}. 
\ena
Since, by Lemma \ref{algebre}, 
\begin{eqnarray*}
\Vert u_0^p - v_0^p \Vert _{L(s')} 
&=& \Vert (u_0-v_0)(u_0^{p-1} +u_0^{p-2}v_0+\cdots + v_0^{p-1} ) \Vert _{L(s')}  \\
&\le& \Vert u_0-v_0\Vert _{L(s')} 
\left( 
\Vert u_0\Vert _{L(s')}^{p-1} + \Vert u_0\Vert _{L(s')} ^{p-2} \Vert v_0\Vert _{L(s')}  + \cdots + 
\Vert v_0\Vert _{L(s')} ^{p-1} 
\right) \\
&\le& p(CDe^{\tau})^{p-1} \Vert u_0-v_0\Vert _{L(s')},
\end{eqnarray*}
we infer that 
\begin{eqnarray}
\Vert F(x,U) - F(x,V) \Vert _{X_{s'}}
&\le& e^{  - \tau (1-s')} \left( \sum_{p+q>0} \vert A_{p,q}(x)\vert p(CDe^\tau )^{p+q-1} \Vert u_0-v_0\Vert _{ L(s') } \right.   \nonumber \\
&&\qquad \qquad \left.  +\sum_{p+q>0} \vert A_{p,q} (x)\vert q(CD e^\tau )^{p+q-1} \Vert u_1-v_1\Vert _{ L(s') }   \right) \nonumber  \\
&=:& e^{-\tau (1-s') }  (S_1+S_2).
\label{ABC4}
\end{eqnarray}
Let us estimate $S_1$. Set $M':=M/(1-b_2^{-1})$. Since
\[
\vert A_{p,q} (x) \vert \le \frac{M}{b_0^pb_1^q} (1-\frac{|x|}{b_2})^{-1} \le \frac{M'}{b_0^pb_1^q} 
\quad \textrm{ for } |x |\le 1, 
\]
we have that 
\begin{eqnarray*}
S_1 &\le& \sum_{p>0} \frac{M'}{b_0} p \left( \frac{CDe^\tau}{b_0} \right) ^{p-1} \sum_{q\ge 0}   \left( \frac{CDe^\tau}{b_1} \right) ^{q}
\Vert u_0-v_0\Vert _{L(s')} \\
&=& \frac{M'}{b_0} \sum_{p\ge 0} (p+1) \left( \frac{CDe^\tau }{b_0}\right) ^p (1-\frac{CDe^\tau}{b_1})^{-1} \Vert u_0-v_0\Vert _{L(s')} \\ 
&=& \frac{M'}{b_0} (1-\frac{CDe^\tau}{b_0})^{-2} (1-\frac{CDe^\tau}{b_1})^{-1} \Vert u_0-v_0\Vert _{L(s')}\\
&\le& \frac{8M'}{b_0} \Vert u_0-v_0\Vert _{L(s')}  \\
&\le& 2 M' \Vert u_0-v_0\Vert _{L(s')}  
\end{eqnarray*} 
provided that 
\be
\label{AB20}
D\le \min ( \frac{ b_0 e^{-\tau} }{2C}, \frac{b_1 e^{-\tau} }{2C}) \cdot
\ee
Similarly, we can prove that 
\[
S_2\le 2M'\Vert u_1-v_1\Vert _{L(s')}. 
\]
Therefore, using \eqref{gainXs}, \eqref{estimexp} and \eqref{TTT1}, we infer that 
\begin{eqnarray*}
\Vert F(x,U)-F(x,V)\Vert _{X_{s'}} 
&\le& 2M' e^{-\tau (1-s')} \left( \Vert u_0-v_0\Vert _{L(s')} +\Vert u_1-v_1\Vert _{L(s')}  \right) \\
&\le& 2M' e^{-\tau (1-s')} \left( C \Vert u_0-v_0\Vert _{L(s'), \frac{1}{2}} +\Vert u_1-v_1\Vert _{L(s')}  \right) \\
&\le& 2CM'  \Vert U-V\Vert _{X_{s'}} \\
&\le& 2CM' e^{-\tau (s-s')} \Vert U-V\Vert _{X_s} \\
&\le& \frac{2CM'}{e} \, \frac{1}{\tau (s-s')} \Vert U-V\Vert _{X_s}   \cdot
\end{eqnarray*}  
To complete the proof of \eqref{R2}, it is sufficient to pick $\tau \ge \tau _0$ with $\tau _0$ such that  $2CM'/(e \tau _0) \le \epsilon $,
 and $D$ as in \eqref{AB20}.
  
For given $0\le s\le 1$ and $U=(u_0,u_1)\in X_s$ with 
$\Vert U\Vert _{X_s} \le D$, let us prove that the map 
$x\in [-1,1]\to F(x,U)\in X_s$ is continuous.   Pick any $x,x'\in [-1,1]$. 
From the mean value theorem, we have for $r\in \N$ that 
$|x^r-x'^r|\le r|x-x'|$ for $r\in \N$, and hence
\[
| A_{p,q}(x)-A_{p,q}(x') |\le |x-x'|\sum_{r\in \N} \frac{rM}{b_0^pb_1^qb_2^r}
=\frac{M}{b_0^pb_1^q}(1-\frac{1}{b_2})^{-2} |x-x'|.
\]
We infer that 
\begin{eqnarray*}
\Vert F(x,U)-F(x',U)\Vert _{X_{s}} 
&=& e^{-\tau (1-s)} \Vert f(x,u_0,u_1)-f(x',u_0,u_1)\Vert _{L(s)} \\
&\le& e^{-\tau (1-s)} \sum_{p+q>0} |A_{p,q}(x)-A_{p,q}(x')| 
\Vert u_0^pu_1^q\Vert _{L(s)} \\
&\le& e^{-\tau (1-s)} 
M(1-\frac{1}{b_2})^{-2} |x-x'|
\sum_{p+q>0} \left( \frac{ \Vert u_0\Vert _{L(s)}}{b_0}\right) ^p 
\left( \frac{\Vert u_1\Vert _{L(s)}}{ b_1}\right) ^q, 
\end{eqnarray*}
the last series being convergent for $\Vert U\Vert _{X_s}\le D$. 
\enp

We are in a position to prove Theorem \ref{thm2}. \\
{\em Proof of Theorem \ref{thm2}}. Let $f=f(x,y_0,y_1)$ be as in \eqref{AB1}-\eqref{AB4}, 
$ - \infty < t_1 < t_2< + \infty$ and $R>4$. Pick $g_0,g_1\in G^2( [t_1,t_2] )$ such that 
\eqref{AA} holds. We will show that Theorem \ref{thm4} can be applied provided that $C$ is small enough.   
Pick $L_1 \in (1/R, 1/4)$. Let $\eta =\eta (L_1)>0$ be as in Theorem \ref{thm4}. Let
$U_0=(u_0,u_1)=(g_0,g_1)$. We have to show that 
\[
\Vert U_0\Vert _{{\mathcal X}_{L_1}}= \Vert g_0\Vert _{L_1,\frac{1}{2}} 
+\Vert g_1\Vert _{L_1} \le \eta
\]
for $C$ small enough.    It is sufficient to have 
\ba
 \Vert g_0\Vert _{L_1,\frac{1}{2}}          &\le& \frac{\eta}{2}, \label{YY1} \\  
\Vert g_1\Vert _{L_1}    &\le& \frac{\eta}{2}.\label{YY2}
\ea
Recall that 
\ba
\Vert g_0\Vert _{L_1\frac{1}{2}} &=& \max \left( 2^6 \Vert g_0\Vert _{L^\infty ( [t_1,t_2] )}, 2^3L_1^{-1} 
\sup_{t\in [t_1,t_2],k\in \N} \frac{\vert g_0^{(k+1)} (t)\vert }{ L_1^{ |k-\frac{1}{2}| } \Gamma _{2,\frac{1}{2}} (k)} \right) , \label{YY3}\\
\Vert g_1\Vert _{L_1} &=& \max \left( 2^6 \Vert g_1\Vert _{L^\infty ( [t_1,t_2] )}, 2^3L_1^{-1} 
\sup_{t\in [t_1,t_2],k\in \N} \frac{\vert g_1^{(k+1)} (t)\vert }{ L_1^k 2^{-5} (k!)^2 (1+k)^{-2}} \right) ,
\label{YY4}
\ea
where 
\[
\Gamma _{2, \frac{1}{2}} (k) = 
\left\{ 
\begin{array}{ll}
2^{-5} \big( \Gamma (k+\frac{1}{2}) \big) ^2 (1+k)^{-2}, &\textrm{ if } k > 3/2, \\
2^{-5} (k!)^2 (1+k)^{-2} &\textrm{ if } 0\le k\le 3/2.
\end{array}
\right. 
\]


Then, if follows that \eqref{YY1} is satisfied provided that 
\begin{eqnarray}
 \Vert g_0\Vert _{L^\infty ( [t_1,t_2] )} &\le& 2^{-7}\eta, \label{YY5}\\
 \Vert g_0 ^{(k+1)} \Vert _{L^\infty ( [t_1,t_2] )} &\le& 2^{-4}\eta   L_1^{1+|k-\frac{1}{2} |} 
\Gamma _{2, \frac{1}{2}} (k) 
  \quad \forall k\in \N.  \label{YY6}
\end{eqnarray}
Since $\Gamma (k+\frac{1}{2})\sim \Gamma (k) k^\frac{1}{2} \sim (k-1)! k^\frac{1}{2}$ as $k\to +\infty$, we have that  $\big( \Gamma ( k+\frac{1}{2} ) \big) ^2\sim (k!)^2/k$. Thus, the r.h.s. of \eqref{YY6}
is equivalent to $2^{-9} \eta L_1 ^{k+\frac{1}{2}}  (k!)^2 k^{-3}$  as $k\to +\infty$. Using 
\eqref{AA} and the fact that $L_1>1/R$, we have that \eqref{YY6} holds if $C$ is small enough. 
The same is true for \eqref{YY5}. 
Similarly, we see that \eqref{YY2} is satisfied provided that 
\begin{eqnarray}
 \Vert g_1\Vert _{L^\infty ( [t_1,t_2] )} &\le& 2^{-7}\eta, \label{YY7}\\
 \Vert g_1 ^{(k+1)} \Vert _{L^\infty ( [t_1,t_2] )} &\le& 2^{-9}\eta   L_1^{k+1} 
(k!)^2 (k+1)^{-2} 
  \quad \forall k\in \N.  \label{YY8}
\end{eqnarray}
Again, \eqref{YY7} and \eqref{YY8} are satisfied if the constant $C$ in \eqref{AA} is small enough.  

We infer from Theorem \ref{thm4} the existence of a solution $U=(y,\partial _x y) \in C([-1,1], X_{s_0})$ 
 for some $s_0\in (0,1)$ of \eqref{C1}-\eqref{C3}. Let us check that $y\in C^\infty ([-1,1]\times [t_1,t_2])$.
To this end, we prove by induction on $n\in \N$ the following statement 
 \be
 \label{PPP1}
 U\in C^n([-1,1], C^k([t_1,t_2])^2)  \quad \forall k\in \N . 
 \ee
The assertion \eqref{PPP1} is true for $n=0$, for $X_{s_0}\subset C^k([t_1,t_2])^2$ for all $k\in \N$. 
 Assume \eqref{PPP1} true for some $n\in \N$. 
Since $A$ is a continuous linear map from $X_s$ to $X_{s'}$ for $0\le s' < s \le 1$, we have that 
\[
AU\in C^n([-1,1],X_s) \subset C^n([-1,1],C^k([t_1,t_2])^2) \quad \forall s\in (0,s_0),  \ \forall k\in \N .
\]
On the other hand, as $f$ is analytic and hence of class $C^\infty$, we infer from \eqref{PPP1} that 
$F(x,U)\in C^n([-1,1],$ $C^k([t_1,t_2])^2 )$ for all $k\in \N$.   Since $\partial _x U=AU+F(x,U)$, we obtain that 
\eqref{PPP1} is true with $n$ replaced by $n+1$. The proof of  
$y\in C^\infty ([-1,1]\times [t_1,t_2])$ is complete.
Finally, the proof of  $y\in  G^{1,2}([-1,1]\times [t_1,t_2])$, which uses some estimates of the next section,  
is given in appendix, with eventually a stronger smallness assumption on the initial data. \qed

\section{Correspondence between the space derivatives and the time derivatives}
\label{section3}

We are concerned with the relationship between the time derivatives and the space derivatives of any solution  of a general nonlinear heat equation 
\be
\label{A1}
\partial _t y = \partial _x ^2 y + f(x,y,\partial _x y),
\ee
where $f=f(x,y_0,y_1)$ is of class $C^\infty$ on $\R ^3$. 

When $f=0$, then the jet $(\partial _x^n y (0,0))_{n\ge 0}$ is nothing but the reunion of the jets  
$(\partial _t^n y (0,0))_{n\ge 0}$ and 
$(\partial _t^n \partial _x y (0,0))_{n\ge 0}$, for
\ba
\partial _t^n y=\partial _x^{2n} y,&& \quad \forall n\in \N ,\label{J1}\\
\partial _t ^n \partial _x y = \partial _x^{2n+1}y,&& \quad \forall n\in \N. \label{J2}  
\ea
When $f$ is no longer assumed to be $0$, then the relations \eqref{J1}-\eqref{J2} do not hold anymore. Nevertheless, there is still a 
one-to-one correspondence between the jet  $(\partial _x^n y (0,0) )_{n\ge 0}$  and the 
jets $(\partial _t^n y (0,0))_{n\ge 0}$ and 
$(\partial _t^n \partial _x y (0,0) )_{n\ge 0}$.

\begin{proposition}
\label{prop1}
Let $-\infty < t_1 \le \tau \le t_2 < +\infty$. 
Assume that $f\in C^\infty (\R ^3)$ and that  $y\in C^\infty ([-1,1]\times [t_1,t_2])$ satisfies \eqref{A1} on $[-1,1]\times [t_1,t_2]$. 
Then the determination of the jet $(\partial _x ^n y(0,\tau))_{n\ge 0}$ is equivalent to the determination of the jets 
$(\partial _t ^n y(0,\tau ))_{n\ge 0}$ and $(\partial _t^n\partial _x y(0,\tau ))_{n \ge 0}$.  
\end{proposition}
\begin{proof}
The proof of Proposition \ref{prop1} is a direct consequence of the following 
\begin{lemma}
\label{lem1}
Let $f\in C^\infty (\R ^3)$ and $n\in \N ^*$. Then there exist two smooth functions $H_n=H_n(x,y_0,y_1,..., y_{2n-1})$ and 
$\tilde  H_n =\tilde H _n (x,y_0,y_1, ... , y_{2n})$ such that any solution $y\in C^\infty ([-1,1]\times [t_1,t_2])$ of \eqref{A1} satisfies 
\ba
\partial _t ^n y &=& \partial _x^{2n} y + H_n(x,y, \partial _x y, ... , \partial _x ^{2n-1}y) \quad \textrm{ for } 
(x,t)\in [-1,1]\times [t_1,t_2], \label{B1}  \\
\partial _t^n\partial _x y &=& \partial _x ^{2n+1} y + \tilde H _n (x,y, \partial _x y, ... ,  \partial _ x ^{2n} y)  \quad \textrm{ for } 
(x,t)\in [-1,1]\times [t_1,t_2]. \label{B2} 
\ea
\end{lemma}
\noindent
{\em Proof of Lemma \ref{lem1}.}  Assume first that $n=1$. Then \eqref{B1} holds with 
$H_1(x,y_0,y_1) =  f(x,y_0,y_1)$. Taking the derivative with respect to $x$ in \eqref{A1}
yields 
\[
\partial _x\partial _t  y = \partial _x ^3 y  +\frac{\partial f}{\partial x} (x,y, \partial _x y) 
+ \frac{\partial f}{ \partial y_0} (x,y, \partial _x y) \partial _x y 
+  \frac{\partial f}{ \partial y_1} (x,y, \partial _x y) \partial _x^2 y,  
\] 
and hence \eqref{B2} holds with $\tilde H_1 (x,y_0,y_1,y_2)= 
\frac{\partial f }{\partial x}(x,y_0,y_1) 
+ \frac{\partial f}{\partial y_0} (x,y_0,y_1) y_1 
+ \frac{\partial f}{\partial y_1} (x,y_0,y_1)  y_2$. 

Assume now that \eqref{B1} and \eqref{B2} are satisfied at rank $n-1$, and let us prove that they are satisfied at rank $n$. 
For \eqref{B1}, we notice that
\begin{eqnarray}
\partial _t ^n y &=& \partial _t (\partial _t ^{n-1} y ) \nonumber \\
&=& \partial _t (\partial _x ^{2n-2} y + H_{n-1} (x,y, \partial _x y , ... , \partial _x ^{2n-3}y)) \nonumber \\
&=& \partial _x ^{2n-2} \partial _t y + \sum_{k=0} ^{2n-3} \frac{ \partial H_{n-1}}{\partial y_k} (x,y, \partial _x y , ... , \partial _x ^{2n-3}y)
\partial _t \partial _x ^k y \nonumber \\
&=& \partial _x ^{2n-2} (\partial _x ^2 y + f (x,y, \partial _x y) ) + \sum_{k=0} ^{2n-3} \frac{ \partial H_{n-1}}{\partial y_k} (x,y, \partial _x y , ... , \partial _x ^{2n-3}y)  \partial _x ^k  (\partial _x ^2 y + f(x,y, \partial _x y))  \label{BBB1} \\
&=:& \partial _x ^{2n} y + H_n (x,y, \partial _x y , ... , \partial _x ^{2n-1} y) \nonumber 
\end{eqnarray}
for some smooth function $H_n=H_n(x,y_0 , ..., y_{2n-1})$. For \eqref{B2}, we notice that 
\begin{eqnarray*}
\partial _t ^n \partial _x y &=& \partial _t (\partial _t ^{n-1} \partial _x y ) \\
&=& \partial _t (\partial _x ^{2n-1} y + \tilde H_{n-1} (x,y, \partial _x y , ... , \partial _x ^{2n-2}y)) \\
&=& \partial _x ^{2n-1} \partial _t y + \sum_{k=0} ^{2n-2} \frac{ \partial \tilde H_{n-1}}{\partial y_k} (x,y, \partial _x y , ... , \partial _x ^{2n-2}y)
\partial _t \partial _x^k y\\
&=& \partial _x ^{2n-1} (\partial _x ^2 y + f (x,y, \partial _x y) ) + \sum_{k=0} ^{2n-2} \frac{ \partial \tilde H_{n-1}}{\partial y_k} (x,y, \partial _x y , ... , \partial _x ^{2n-2}y)  \partial _x ^k  (\partial _x ^2 y + f(x,y, \partial _x y)) \\
&=:& \partial _x ^{2n+1} y + \tilde H_n (x,y, \partial _x y , ... , \partial _x ^{2n} y)
\end{eqnarray*}
for some smooth function $\tilde H_n=\tilde H_n (x,y_0, y_1, ... , y_{2n})$.  
\end{proof}

Next, we relate the behaviour as $n\to +\infty$ of the jets
$(\partial _t ^n y(0,\tau ))_{n\ge 0}$ and $(\partial _t^n\partial _x y(0,\tau ))_{n \ge 0}$  to those 
of the jet $(\partial _x ^n y(0,\tau ))_{n\ge 0}$. To do that, we assume that in \eqref{A1} the nonlinear term reads 
\be
\label{AAA}
f(x,y_0,y_1)= \sum_{ (p,q,r)\in \N^3} a_{p,q,r} (y_0)^p (y_1)^q x^r  \quad \forall (x,y_0,y_1)\in (-4,4)^3, 
\ee
where the coefficients $a_{p,q,r}$, $(p,q,r)\in \N ^3$, satisfy \eqref{AB3}-\eqref{AB4}. 
\begin{proposition}
\label{prop10}
Let $-\infty < t_1 \le \tau \le t_2 < +\infty$ and $f=f(x,y_0,y_1)$ 
be as in  \eqref{AB1}-\eqref{AB2} with the 
coefficients  $a_{p,q,r}$, $(p,q,r)\in \N ^3$, satisfying \eqref{AB3}-\eqref{AB4}.
Pick any $\widetilde R>4$ and any numbers $R,R'$ with $4<R'<R<\min (\widetilde R,b_2)$. Then there exists some number $\widetilde C>0$ such that for any 
$C\in (0,\widetilde C ]$, ,
we can find a number $C' =C'(C,R,R')>0$ with $\lim _{C\to 0^+} C'(C,R,R') =0$ such that for any function 
 $y\in C^\infty([-1,1]\times [t_1,t_2] )$ satisfying
 \eqref{A1} on $[-1,1]\times [t_0,t_1]$ and 
\be
y(x,\tau )=y^0(x)=\sum_{n=0}^\infty a_n \frac{x^n}{n!} , \quad \forall x\in [-1,1] \label{D1}\\
\ee
for some $y^0\in {\mathcal R}_{\tilde R,C}$,
is such that
\be
\vert \partial _x^k\partial _t ^n y(0,\tau)\vert \le C' \frac{(2n+k)!}{R^k R'^{2n}}, \quad \forall k,n\in \N .\label{D2} 
\ee
In particular, we have 
\ba
\vert \partial _t ^n y(0,\tau)\vert &\le& C' \frac{(2n)!}{R'^{2n}}, \quad \forall n\in \N ,\label{D21}\\ 
\vert \partial _x\partial _t ^n y(0,\tau)\vert &\le& C' \frac{(2n+1)!}{R R'^{2n}}, \quad \forall n\in \N .\label{D22} 
\ea
\end{proposition}

\begin{proof}
We know from Proposition \ref{prop1} that the jets 
$(\partial _t ^n y(0, \tau ))_{n\ge 0} $ and 
$(\partial _t ^n\partial _x y(0,\tau ))_{n\ge 0}$ are completely determined by the jet 
$(\partial _x ^n y(0,\tau ))_{n\ge 0}$, that is by $y^0$. A {\em direct} proof of estimates \eqref{D21} and \eqref{D22}  (which follow at once from \eqref{D2}) seems 
hard to be derived, whereas  
a proof of \eqref{D2} can be obtained by induction on $n$. We shall need several lemmas. 
\begin{lemma}
\label{lem2}
(see \cite[Lemma A.1]{KiNi}) For all $k,q\in \N$ and $a\in \{ 0 , ... , k+q\}$, we have 
\[
\sum_{\tiny
\begin{array}{c} 
j+p=a\\
0\le j \le k\\
0\le p\le q
\end{array}}
\left( \begin{array}{c} k\\j \end{array} \right) \, 
\left( \begin{array}{c} q\\p \end{array} \right)  =
\left( \begin{array}{c} k + q \\a \end{array} \right) . 
\]
\end{lemma}
The following Lemma gives  the algebra property for the mixed Gevrey spaces $G^{1,2}([-1,1]\times [t_1,t_2])$. A slight modification of its proof actually yields Lemma \ref{algebre}, making the paper almost self-contained.
\begin{lemma}
\label{lem3}
Let $(x_0,t_0)\in [-1,1]\times [t_1,t_2]$, $R,R' \in (0,+\infty ) $, $q\in \N$, $\mu  \in (q+2,+\infty )$,  
 $k_0,n_0\in \N$, $C_1,C_2\in (0,+\infty )$, and $y_1,y_2\in C^\infty ([-1,1]\times [t_1,t_2] )$ be such that 
\be
\vert \partial _x ^k \partial _t ^n y_i (x_0,t_0)\vert \le C_i \frac{ (2n+k+q)!}{R^k R'^{2n} (2n+k+1)^\mu  } \quad \forall i=1,2, \quad  \forall k\in \{ 0, ..., k_0\} ,\ \forall n\in \{ 0, ..., n_0 \} .  
\label{A49}
\ee
Then we have 
\be
\vert \partial _x^k \partial _t ^n (y_1y_2) (x_0,t_0) \vert \le K_{q,\mu}  C_1C_2 \frac{ (2n+k+q )!}{R^k R'^{2n} (2n+k+1)^\mu } \quad  \forall k\in \{ 0, ..., k_0\} ,\ \forall n\in \{ 0, ..., n_0 \} ,
 \label{A50}
\ee
where
\[
K_{q,\mu} :=2^{\mu -q} (1+q)^{2q} \sum_{j\ge 0} \sum_{i\ge 0} \frac{1}{ (2 i +j+1)^{\mu -q} } <\infty  .
\]
\end{lemma}
\noindent 
{\em Proof of Lemma \ref{lem3}:} 
Using $(2n+k+q)^q\leq (1+q)^q\left(1+2n+k\right)^q$, we obtain 
\[
(2n+k+q)!\leq (2n+k)! (2n+k+q )^q\leq (1+q)^q(2n+k)!\left(1+2n+k\right)^q.
\]
So, denoting $\mub :=\mu-q>2$ and $\widetilde{C}_i:=(1+q)^qC_i$, we have
\be
\vert \partial _x ^k \partial _t ^n y_i (x_0,t_0)\vert \le \widetilde{C}_i \frac{ (2n+k)!}{R^k R'^{2n} (2n+k+1)^{\mub}  } \quad \forall i=1,2, \quad  \forall k\in \{ 0, ..., k_0\} ,\ \forall n\in \{ 0, ..., n_0 \} .  
\ee
From Leibniz' rule, we have that 
\begin{eqnarray*}
&& \vert \partial _x ^k \partial _t ^n (y_1y_2) (x_0,t_0) \vert \\
&&\qquad = 
\left\vert \sum _{0\le j\le k}  \,  \sum_{0\le i\le n}  \left( \begin{array}{c} k\\j \end{array} \right) \, 
\left( \begin{array}{c} n\\i \end{array} \right)  
(\partial _x^j\partial_t^i y_1) (x_0,t_0)  (\partial _x^{k-j} \partial _t ^{n-i} y_2) (x_0,t_0)  \right\vert \\
&&\qquad \le 
\sum _{0\le j\le k}  \,  \sum_{0\le i\le n}  \left( \begin{array}{c} k\\j \end{array} \right) \, 
\left( \begin{array}{c} n\\i \end{array} \right)  
\widetilde{C}_1 \frac{(2i +j)!}{R^jR'^{2i} (2i+j+1)^{\mub} }\widetilde{C}_2 \frac{(2(n-i) +k-j )!}{R^{k-j}R'^{2(n-i)}
(2(n-i)+k-j+1)^{\mub}} \\
&&\qquad = \frac{\widetilde{C}_1\widetilde{C}_2}{R^kR'^{2n}} (2n+k) !   
\underbrace{\sum _{0\le j\le k}  \,  \sum_{0\le i\le n} 
\frac{ \left( \begin{array}{c} k\\j \end{array} \right) \, \left( \begin{array}{c} n\\i \end{array} \right) \, 
 \left( \begin{array}{c} 2n+k\\ 2i+j   \end{array} \right) ^{-1} } 
 {(2i+j+1)^{\mub}(2(n-i)+k-j+1)^{\mub}}}_{I} \cdot
\end{eqnarray*}
We infer from Lemma \ref{lem2} that 
\be
\label{AB0}
\left( \begin{array}{c} k\\j \end{array} \right) \, \left( \begin{array}{c} q\\p \end{array} \right)
\le \left( \begin{array}{c} k+q \\j  + p\end{array} \right), \quad \textrm{ for  } 0\le j\le k, \ 0\le p\le q.    
\ee
This yields
\[
\left( \begin{array}{c} n\\i \end{array} \right) 
\le \left( \begin{array}{c} n\\i \end{array} \right) ^2
\le \left( \begin{array}{c} 2n\\ 2i \end{array} \right) ,
\]
and hence (using again \eqref{AB0}) 
\[
\left( \begin{array}{c} k\\j \end{array} \right) \, \left( \begin{array}{c} n\\i \end{array} \right) 
\le \left( \begin{array}{c} k\\j \end{array} \right) \, \left( \begin{array}{c} 2n\\ 2i \end{array} \right)
\le\left( \begin{array}{c} 2n+ k\\ 2i+ j \end{array} \right) \cdot 
\]
Finally, by convexity of $x\to x^{\mub}$, we have that 
 \begin{eqnarray*}
&&\sum _{0\le j\le k}  \sum_{0\le i\le n}  \frac{ (2n+k+1)^{\mub}}{(2i+j+1)^{\mub}(2(n-i)+k-j+1)^{\mub}} 
= \sum _{0\le j\le k}\sum_{0\le i\le n}  \left( \frac{1}{ 2i+j+1}  + \frac{1}{2(n-i)+k-j+1}  \right)^{\mub} \\
&\le& 2^{\mub -1} \sum _{0\le j\le k}\sum_{0\le i\le n}  \left( \frac{1}{ (2i+j+1)^{\mub} }  + \frac{1}{ (2(n-i)+k-j+1)^{\mub} }  \right)  \le 2^{\mub} \sum_{j\ge 0}\sum_{i\ge 0} \frac{1}{ (2i+j+1)^{\mub} } <\infty ,
\end{eqnarray*}
where we used the fact that ${\mub} =\mu  - q >2$. 

It follows that 
\[ 
I\le 2^{\mub} 
\left( \sum_{j\ge 0} \sum_{i\ge 0} \frac{1}{ (2i+j+1)^{\mub}} \right) \frac{1}{ (2n+k+1)^{\mub}}
=2^{\mu -q} \left( \sum_{j\ge 0}\sum_{i\ge 0} \frac{1}{ (2i+j+1)^{\mu -q}} \right) 
\frac{  (2n+k+1)^q }{ (2n+k+1)^\mu },
\] 
and hence the proof of Lemma \ref{lem3} is complete once we have
noticed that $ (2n+k) !(2n+k+1)^{q}\leq  (2n+k+q) !$. \qed


Let us go back to the proof of Proposition \ref{prop10}. Pick any number $\mu >3$. We shall prove by induction on 
$n\in \N$ that 
\be
\vert \partial _x^k\partial _t ^n y(0,\tau)\vert \le C_n \frac{(2n+k)!}{R^k R'^{2n} (2n+k+1)^\mu }, \quad  \forall k\in \N ,\label{D20} 
\ee
where $0<C_n \le C'< +\infty$. 
For $n=0$, using the fact that $R<\widetilde R$, we have that 
\[
\vert \partial _x^k y(0,\tau)\vert  
= \vert  a_k \vert  
\le  C\frac{k ! }{ \widetilde R ^k} 
\le  C \left(\sup_{p\in \N} (\frac{R}{\widetilde R})^p (p+1)^\mu \right) \frac{k ! }{ R^k(k+1)^\mu } 
\le  C_0\frac{k ! }{ R^k (k+1)^\mu } 
\]
provided that 
\be
\label{cond1}
C\le\widetilde C = \left(\sup_{p\in \N} (\frac{R}{\widetilde R })^p (p+1)^\mu \right) ^{-1} C_0.
\ee
Assume that \eqref{D20} is satisfied at the rank $n\in \N$ for some constant  $C_n>0$. 
Then, by \eqref{W1}, \eqref{AB2}, we have that 
\begin{eqnarray}
\vert \partial _x^k \partial _t ^{n+1} y(0,\tau ) \vert 
&=& \vert \partial _x ^k \partial _t ^n 
\big( \partial _x^2 y + \sum_{p,q,r\in \N} 
a_{p,q,r} y^p(\partial _x y)^qx^r \big) (0,\tau)\vert \nonumber \\
&=& \vert \partial _x ^k \partial _t ^n 
\big( \partial _x^2 y + \sum_{p,q\in \N} 
A_{p,q}(x) y^p(\partial _x y)^q \big) (0,\tau)\vert \nonumber \\
&\le& \vert \partial _t^n \partial _x ^{k+2} y(0,\tau )\vert
+\sum_{p\ge 1} 
\vert \partial _x^k\partial _t^n \big( A_{p,0}(x) y^p\big) (0,\tau)\vert \nonumber \\
&&\quad 
+
\sum _{q\ge 1}\sum_{p\ge 0} 
\vert \partial _x^k\partial _t^n \big( A_{p,q}(x) y^p (\partial _x y)^q \big) (0,\tau)\vert \nonumber \\
&=:& I_1+I_2+I_3. \label{E1}
\end{eqnarray}
(Note that the sum for $I_2$ is over $p\ge 1$, for $A_{0,0}(x)=0$.)

Since $R'<R$, we can pick some $\varepsilon \in (0,1)$ such that
\[
R'^2\le (1-\varepsilon) R^2. 
\]
For $I_1$, we have that 
\be
I_1 \le C_n\frac{(2n+k+2)!}{R^{k+2} R'^{2n} (2n+k+3)^\mu   }\le 
(1-\varepsilon ) C_n\frac{(2n+k+2)!}{R^{k} R'^{2n+2} (2n+k+3)^\mu} \cdot
\label{E2}
\ee

Since $A_{p,q}$ does not depend on $t$, we have that 
$\partial _x^k\partial _t^nA_{p,k}=0$ for $n\ge 1$ and $k\ge 0$. Next, for $k\ge 0$, we have that 
\[
\vert \partial _x ^k A_{p,q}(0)\vert =k!\, \vert a_{p,q,k}\vert \le \frac{k !}{b_2^k}\, \frac{M}{b_0^p b_1^q}
\le \frac{\overline{C}\, k!}{(k+1)^\mu R^k b_0^p b_1^q}
\]
for some constant $\overline{C}>0$ depending on $R$, $b_2$, $\mu$, for $R<b_2$.  

Note that, still by  \eqref{D20}, the function $\partial _x y$ satisfies the estimate
\[
\vert \partial _x^k\partial _t ^n (\partial _x  y)(0,\tau)\vert \le \frac{C_n}{R}
 \frac{(2n+k +1)!}{R^k R'^{2n} (2n+k+2)^\mu }, \quad  \forall k\in \N . 
\]

Using $\mu -1>2$, we infer from iterated applications of Lemma \ref{lem3}  that 
\ba
\left\vert 
\partial_x ^k\partial_t^n \big( A_{p,0}y^p\big) (0, \tau ) 
\right\vert 
&\le& \frac{\overline{C} C_n^p K^{p} (2n+k)!}{R^k{R'}^{2n} (2n+k+1)^\mu b_0^p} ,
\label{E3zero}\\
\left\vert 
\partial_x ^k\partial_t^n \big( A_{p,q}y^p(\partial_x y)^q\big) (0, \tau ) 
\right\vert 
&\le& \frac{\overline{C} C_n^{p+q} K^{p+q} (2n+k+1)!}{R^k{R'}^{2n} (2n+k+1)^\mu b_0^pb_1^qR^q}
\quad \forall q\ge 1, 
\label{E3}
\ea
where we denote $K:=\max (K_{0, \mu }, K_{1, \mu })$.
We infer from \eqref{E3zero}-\eqref{E3}  that
\ba
I_2 &\le& \sum_{p\ge 1} \frac{\overline{C} C_n^p K^p (2n+k)!}{R^k{R'}^{2n} (2n+k+1)^\mu b_0^p}, \label{E4}\\
I_3 &\le& \sum_{q\ge 1} \sum_{p\ge 0} \frac{\overline{C} C_n^{p+q} K^{p+q} (2n+k+1)!}{R^k{R'}^{2n} (2n+k+1)^\mu 
b_0^pb_1^qR^q}\cdot  \label{E5}
\ea
Using \eqref{E1}-\eqref{E2} and \eqref{E4}-\eqref{E5}, we see that  the condition 
\[
\vert \partial _x^k\partial _t ^{n+1} y(0,\tau)\vert \le C_{n+1} \frac{(2n+k+2)!}{R^k R'^{2n+2} 
(2n+k+3)^\mu }, \quad  \forall k\in \N ,
\]
is satisfied provided that 
\begin{multline}
(1-\varepsilon)C_n
+\sum_{p\ge 1} \frac{\overline{C}R'^2}{(2n+k+1)(2n+k+2)} \left( \frac{2n+k+3}{2n+k+1}\right) ^\mu  (\frac{C_nK}{b_0})^p \\
+\sum_{q\ge 1}\sum_{p\ge 0} \frac{\overline{C}R'^2}{(2n+k+2)} 
\left( \frac{2n+k+3}{2n+k+1}\right) ^\mu   (\frac{C_nK}{b_0})^p  (\frac{C_nK}{b_1R})^q   
\le 
C_{n+1}\cdot \label{E100}
\end{multline}
Pick a number $\delta\in (0,1)$. Assume that
\be 
C_n\le \delta \cdot \min (\frac{b_0}{K}, \frac{b_1R}{K}),
\label{E150}
\ee 
so that $C_nK/b_0\le \delta$ and $C_nK/(b_1R)\le \delta$.
Set
\begin{eqnarray*}
C_{n+1}&=& \lambda _nC_n:= [(1-\varepsilon )   + \frac{K}{b_0} 
\frac{\overline{C}R'^2}{(2n+1)(2n+2)} \frac{3^\mu }{1-\delta} 
+ \frac{K}{b_1R} \frac{\overline{C}R'^2}{(2n+2)} \frac{3^\mu}{(1-\delta )^2}  ]C_n  \cdot
\end{eqnarray*}
Then, with this choice of $C_{n+1}$,  \eqref{E100} holds provided that \eqref{E150} is satisfied. 
Next, one can pick some $n_0\in \N$ such that for
$n\ge n_0$,  we have 
\[
\lambda _n\le 1 .
\]  
This yields $C_{n+1}\le C_n$ for $n\ge n_0$, provided that \eqref{E150} holds for $n=n_0$. To ensure \eqref{E150} for $n=0,1,..., n_0$, it is sufficient to choose $C_0$ small enough (or, equivalently, $\widetilde C$ small enough) so that 
\[
\max \big( C_0, \lambda _0C_0, \lambda_1\lambda _0C_0, ..., \lambda_{n_0-1}\cdots \lambda _0
C_0 \big) \le \delta \cdot \min ( \frac{b_0}{K}, \frac{b_1R}{K}) \cdot
\] 
The proof by induction of \eqref{D20} is achieved. 

We can pick 
\[
C'(C,R,R') := \max \big( C_0, \lambda _0C_0, \lambda_1\lambda _0C_0, ..., \lambda_{n_0-1}\cdots \lambda _0\, C_0 \big)
\] 
with 
$C_0=C \sup_{p\in \N} (\frac{R}{\widetilde R })^p (p+1)^\mu$, so that $C'(C,R,R')\to 0$ as $C\to 0$.  
The proof of Proposition \ref{prop10} is complete. 
\end{proof}


\begin{proposition}
\label{prop100}
Let $-\infty < t_1\le \tau \le t_2<+\infty$ and $f=f(x,y_0,y_1)$ be as in \eqref{AB1}-\eqref{AB2} with the coefficients $a_{p,q,r}$, $(p,q,r)\in \N ^3$,
satisfying \eqref{AB3}-\eqref{AB4}.  Assume in addition that $b_2> \hat R := 4e^{(2e)^{-1}}\approx 4.81$.
Let $\tilde R>\hat R$. 
Then there exists some number $\widetilde C>0$ such that 
for any $C\in (0,\widetilde C ]$ and any numbers $R,R', L$ with $\hat R<R'<R<\min (\tilde R, b_2)$ and $4e^{e^{-1}} /R'^2 < L< 1/4$,  there exists a number $C''=C''(C,R,R', L)>0$ with 
$\lim_{C\to 0^+} C''(C,R,R',L)=0$ such that for any $y^0\in {\mathcal R} _{\tilde R,C}$, we can pick a function $y\in G^{1,2}( [-1,1]\times [t_1,t_2] )$ 
satisfying  \eqref{A1} for $(x,t)\in [-1,1]\times [t_1,t_2]$ and 
\be
\label{F1}
y(x,\tau ) =y^0(x)=\sum_{n=0}^\infty a_n \frac{x^n}{n!}, \quad \forall x\in [-1,1], 
\ee
and such that for all $t\in [t_1,t_2]$ 
\ba
\label{F2}
\left\vert 
\partial _t^n y(0, t)
\right\vert \le C''  L ^n (n!)^2, \\
\label{F3}
\left\vert 
\partial _x \partial _t^n y(0, t)
\right\vert \le C'' L^n (n!)^2.
\ea  
\end{proposition}
\begin{proof}
Let $\hat R:=4 e^{(2e)^{-1}}$, $\tilde R>\hat R$ and $R,R'$ with $\hat R<R'<R <\min (\tilde R, b_2)$. 
Pick $\widetilde C,C$ as in Proposition \ref{prop10}, and pick any $y^0\in {\mathcal R} _{\tilde R,C}$.
If a function $y$ as in Proposition \ref{prop100} does exists, then both sequences of numbers 
\begin{eqnarray*}
d_n &:=& \partial _t^n y(0,\tau ), \quad n\in \N, \\
\tilde d_n &:=& \partial _x \partial _t^n y (0,\tau ), \quad n\in \N
\end{eqnarray*}
can be computed inductively in terms of the coefficients $a_n=\partial _x ^n y^0 (0)$, $n\in \N$, according to Proposition \ref{prop1}. 
Furthermore, it follows from Proposition \ref{prop10} (see \eqref{D21}-\eqref{D22})   that we have for some $C'=C'(C,R,R')>0$ and all 
$n\in \N$ 
\begin{eqnarray*}
|d_n|&\le& C'\frac{(2n)!}{R'^{2n}}, \\
|\tilde d_n | &\le& C'\frac{(2n+1)!}{ R R'^{2n}}. 
\end{eqnarray*}

Note that both sequences $(d_n)_{n\in \N}$ and $(\tilde d_n)_{n\in \N}$ (as above) can be defined in terms of the coefficients $a_n$'s, 
even if the existence of the function $y$ is not 
yet established. 

Let $L\in (\frac{4e^{e^{-1}}}{R'^2}, \frac{1}{4})$,  $R''\in (\sqrt{ \frac{4e^{e^{-1} }} {L} }, R')$
and $M=M(R,R',R'') >0$ such that 
\[
|\tilde d_n | \le M C'\frac{(2n)!}{ R''^{2n}}, \quad \forall n\in \N. 
\] 
The following lemma is a particular case of \cite[Proposition 3.6]{MRRreachable} (with $a_p=[2p(2p-1)]^{-1}$ for 
$p\ge1$).  
\begin{lemma}
\label{lem4}
Let $(d_q)_{q\ge 0}$ be a sequence of real numbers such that  
\[
|d_q|\le C H^k (2q)!\quad \forall q\ge 0  
\]
for some $H>0$ and $C>0$. Then for all $\tilde H>e^{e^{-1}}H$ there exists a function $f\in C^\infty(\R )$ such that 
\ba
f^{ (q) } (0) &=& d_q\quad \forall q\ge 0, \\
| f^{ (q) } (t) |&\le&  C \tilde H^q (2q)! \quad \forall q\ge 0, \ \forall t\in \R .  
\ea 
\end{lemma}
Pick $H=1/R''^2$ and $\tilde H=L/4>e^{e^{-1}}H$.
Then by Lemma \ref{lem4}, there exist two functions $g_0,g_1\in G^2( [-1,1] )$
such that 
\ba
g_0^{(n)} (0) &=& d_n, \quad n\ge 0,\label{K1}\\
g_1^{(n)} (0) &=& \tilde d_n, \quad n\ge 0, \label{K2}\\
|g_0^{(n)} (t) | &\le&  C' \tilde H^n (2n)!, \quad n\ge 0,\ t\in [-1,1], \label{K3}\\
|g_1^{(n)} (t) | &\le& M C' \tilde H^n  (2n)! , \quad n\ge 0,\ t\in [-1,1]. \label{K4}
\ea 
It follows at once from Stirling' formula that $(2n)!\le C_s 4^n (n!)^2$ for some universal constant $C_s>0$, so that (with $4\tilde H=L$)
\ba
|g_0^{(n)} (t) | &\le&  C' C_s(4\tilde H)^n (n!)^2, \quad n\ge 0,\ t\in [-1,1], \label{K33}\\
|g_1^{(n)} (t) | &\le& M C' C_s (4\tilde H)^n  (n!)^2 , \quad n\ge 0,\ t\in [-1,1]. \label{K44}
\ea 
Note that $4\tilde H =L<1/4$. 
If $C$ is sufficiently small, then $C'$ is as small as desired, and 
it follows then from Theorem \ref{thm2}  that we can pick a function $y\in G^{1,2}([-1,1]\times [t_1,t_2])$ satisfying 
\eqref{C1}-\eqref{C3}. Using again Proposition \ref{prop1}, we infer that 
$\partial _x^k y (0,\tau )= a_k= \partial _x^k y^0(0)$ for all $k\ge 0$, and hence \eqref{F1} holds. 
The estimate \eqref{F2}-\eqref{F3}  follow from \eqref{C2}-\eqref{C3} and \eqref{K33}-\eqref{K44} with $L=4\tilde H$
and $C''=\max (C'C_s, MC'C_s)$.  
The proof of  Proposition \ref{prop100}  is complete.
\end{proof}

\section{Proofs of the main results.}
\label{section4}
Let us start with the proof of Theorem \ref{thm1}. 
Let $R>\hat R = 4e^{(2e)^{-1}}$ and let $\hat C$ be (for the moment) the constant $\tilde C$ given by Proposition \ref{prop100}. Pick any $y^0,y^1\in {\mathcal R} _{R, \hat C}$. 
We infer from Proposition \ref{prop100}  applied with $[t_1,t_2] =[0,T]$ and $\tau =0$ (resp. $\tau =T$) the existence of 
two functions $\hat y,\tilde y \in G^{1,2}([-1,1]\times [0,T])$ satisfying \eqref{C1} 
and such that 
\[ \hat y(x,0)=y^0(x) \ \textrm{ and } \  \tilde y(x,T)=y^1(x), \quad \forall x\in [-1,1].\]
Let $\rho\in C^\infty (\R )$ be such that
\[
\rho (t) =\left\{
\begin{array}{ll}
1 & \textrm{ if } t \le \frac{T}{4} , \\[3mm]
0 & \textrm{ if } t \ge \frac{3T}{4} ,  
\end{array}
\right.
\]
and $\rho _{\vert [0,T]}  \in G^{\frac{3}{2}}([0,T])$. 
Let 
\begin{eqnarray*}
g_0(t)&=& \rho (t) \hat y (0,t) + (1-\rho (t)) \tilde y(0,t), \quad t\in [0,T], \\
g_1(t)&=& \rho (t) \partial _x \hat y (0,t) + (1-\rho (t)) \partial _x \tilde y(0,t), \quad t\in [0,T].
\end{eqnarray*}
Then by \cite[Lemma 3.7]{MRRreachable} $g_0,g_1\in G^2([0,T])$ and using \eqref{F2}-\eqref{F3} and picking a smaller value of $\hat C$ if necessary, we can assume 
that \eqref{AA} is satisfied with $R=1/L$. It follows then from Theorem \ref{thm2} that there exists a solution $y\in G^{1,2}([-1,1]\times [0,T])$
of \eqref{C1}-\eqref{C3}. The control inputs $h_{-1}$ and $h_1$ are defined by using \eqref{W2}-\eqref{W3}. Then $y$ satisfies \eqref{W1}-\eqref{W4} together with $y(x,T)=y_1(x)$ for $x\in [-1,1]$.  
Indeed, since $\rho (t)=0$ for $t>3T/4$, we have 
\begin{eqnarray*}
\partial _t ^n y(0, T)=g_0^{ (n) } (T ) &=&\partial _t^n \tilde y(0,T),\quad \forall n\in \N, \\
\partial _x \partial _t ^n y(0, T)=g_1^{ (n) } (T) &=&\partial _x \partial _t^n \tilde y(0,T),\quad \forall n\in \N .
\end{eqnarray*}
It follows then from Proposition \ref{prop1} that  $\partial _x ^ny(0,T)=\partial _x ^n \tilde y (0,T)=\partial _x^n y^1(0)$ 
for all $n\in \N$, and hence $y(.,T)=y^1$. The proof of \eqref{W4} is similar.
The proof of Theorem \ref{thm1} is achieved.\qed \\

Let us now proceed to the proof of Corollary \ref{cor1}. Pick any solution $y=y(x,t)$ for  $x\in [-1,1]$ and $t\in [t_1,t_2]$  of \eqref{A1}, and 
set $y^0(x)=y(x,\tau )$ where $\tau\in [t_1,t_2]$. Assume that $y^0(-x)=-y^0(x)$ for $x\in [-1,1]$. The following claims are needed. \\[3mm]
{\sc Claim 1}. For all $n\ge 0$, there exists a smooth function $\widehat H_n$ such that we have $\partial _x^n [\partial _x^2y + f(x,y,\partial _x y) ] =\widehat H_n (x,y,\partial _x y , ...,  \partial _x ^{n+2}y)$,
where 
\[
\widehat H _n (-x,-y_0,y_1,-y_2, ..., (-1)^{n+1} y_{n+2} ) = (-1)^{n+1}\widehat H_n (x,y_0,y_1, ..., y_{n+2}).  
\] 
The proof is by induction on $n\ge 0$. 
Claim 1 is obvious for $n=0$ (take $\widehat H_0(x,y_0,y_1,y_{2})=y_{2}+f(x,y_0,y_1)$), 
and if it is true for some $n\in \N$, then 
\begin{eqnarray*}
\partial _x^{n+1} [\partial _x^2y + f(x,y,\partial _x y) ] 
&=& \partial _x \partial _x^n [\partial _x^2y + f(x,y,\partial _x y) ] \\
&=&\partial _x  [\widehat H_n (x,y,\partial _x y , ...,  \partial _x ^{n+2}y)]\\
&=& \partial _x \widehat H_n (x,y,\partial _x y , ...,  \partial _x ^{n+2} y) 
+ \sum_{k=0}^{n+2} \partial _{y_k} \widehat H_n (x,y,\partial _x y , ...,  \partial _x ^{n+2}y) \partial _x^{k+1}y\\
&=:& \widehat H_{n+1} (x,y,\partial _x y,..., \partial _x^{n+2} y, \partial _x ^{n+3} y). 
\end{eqnarray*}
Then it can be seen that 
\[
\widehat H _{n+1} (-x,-y_0,y_1,-y_2, ..., (-1)^{n+1} y_{n+2} , (-1)^{n+2} y_{n+3}) = (-1)^{n+2}\widehat H_{n+1} (x,y_0,y_1, ..., y_{n+3}).  
\] 
Our second claim is concerned with the function $H_n$ in Lemma \ref{lem1}. \\[3mm]
{\sc Claim 2}. For all $n\ge 1$ we have $H_n ( -x, -y_0 , y_1 ,....,  (-1)^{2n-1}y_{2n-1} ) = - H_n(x , y_0 , y_1 ,...,  y_{2n-1})$. \\
We prove Claim 2 by induction on $n\ge 1$. For $n=1$, the result is obvious, for $H_1(x,y_0,y_1)=f(x,y_0,y_1)$. Assume the result true at rank $n-1 \ge 1$. Then  we infer from \eqref{B1} and \eqref{BBB1} that 
\begin{eqnarray*}
H_n(x,y,\partial _x y, ..., \partial _x ^{2n-1} y)&=& \partial _x ^{2n-2}  [ \partial _x^2 y  +  f(x,y,\partial _xy)] -\partial _x^{2n} y \\
&&\quad +\sum_{k=0}^{2n-3} 
\frac{\partial H_{n-1}}{\partial y_k} (x,y,\partial_x y , ... , \partial _x ^{2n-3} y)\partial _x ^k (\partial _x^2 y + f(x,y,\partial _x y))\\
&=&  \widehat H_{2n-2} (x,y, \partial _x y , ...,   \partial _x ^{2n}  y) -\partial _x ^{2n} y \\
&&\quad +\sum_{k=0}^{2n-3} 
\frac{\partial H_{n-1}}{\partial y_k} (x,y,\partial_x y , ... , \partial _x ^{2n-3} y) 
\widehat H_k (x,y, \partial _x y , ..., \partial _x^{k+2} y).
\end{eqnarray*}
Using Claim 1 and the induction hypothesis, one readily sees that
\[
H_n ( -x, -y_0 , y_1 ,....,  (-1)^{2n-1}y_{2n-1} ) = - H_n(x , y_0 , y_1 ,...,  y_{2n-1}). 
\]
Claim 2 is proved. \\[3mm]
{\sc Claim 3}. $\partial _t^n y(0,\tau)=0\ \  \forall n\in \N $. \\
Note that the result is true for  $n=0$, for $y(0,\tau )=y^0(0)=0$. By Claim 2, we have 
\begin{eqnarray*}
\partial _t^n y (0, \tau ) 
&=&\partial _x ^{2n} y (0, \tau ) +H_n(0, y(0, \tau),\partial _x y(0, \tau),  ..., \partial _x^{2n-1} y(0, \tau ))\\
&=& \partial _x^{2n} y^0(0) + H_n(0, y^0(0), \partial _x y^0(0), ..., \partial _x ^{2n-1} y^0(0)).  
\end{eqnarray*}
It is clear that the function $\partial _x^{2n} y^0$ is odd, and it follows from Claim 2 that the function 
\[ x\to H_n(x, y^0(x), \partial _x y^0(x), ..., \partial _x ^{2n-1} y^0(x))\] is odd as well. It follows that 
$\partial _t^n y (0, \tau )=0$.  The proof of Claim 3 is achieved. \\
Let us go back to the proof of Corollary \ref{cor1}. Let us show that $\hat y(0,t)=\tilde y(0,t)=0$ for all $t\in [0,T]$.
Let us consider $\hat y(0,t)$ only, the property for $\tilde y (0,t)$ being similar. The function $\hat y$ is given by Proposition \ref{prop100}. But in the proof of Proposition \ref{prop100}, as $d_n=\partial _t y^n (0, \tau)=0$ for all $n\in \N$, it is sufficient to pick $g_0(t)=0$ for all $t\in [0,T ] $, so that 
$\hat y(0,t)=0$ for $t\in [0,T]$. 
Finally, the function $y=y(x,t)$ for $(x,t)\in [-1,1]\times [0,T]$ given by Theorem \ref{thm1} yields by restriction to $[0,1]\times [0,T]$ 
the solution of the control problem  \eqref{WW1}-\eqref{WW4}. \qed
  
\section*{Appendix: Gevrey regularity of the solution of \eqref{C1}-\eqref{C3} provided in Theorem \ref{thm2}}
Assume that $f$ satisfies \eqref{AB1}-\eqref{AB4}.
Let us show that $y\in G^{1,2}([-1,1]\times [t_1,t_2])$. Pick any numbers $R_1,R_2$ such that 
$4/e<R_1<R_2$, and let us prove that there exists some constant $M>0$ such that \eqref{AAAAA} holds. 
To this end, picking any $\mu >3$, we prove by induction on $l\in \N$ that 
\ba
\vert \partial _x^k \partial _t ^n y(x,t) \vert &\le& C_l \frac{ (2n+k)!}{R_1^kR_2^{2n} (2n+k+1)^\mu} 
\quad \forall (x,t)\in [-1,1]\times [t_1,t_2], \ \forall n \in \N , 
\label{TT1}
\ea
for $l\in \N$ and $k\in \{2l , 2l+1 \}$, with $\sup_{l\in \N} C_l <\infty$.  
Let us start with $l=0$. Then \eqref{TT1} reads
\begin{eqnarray}
\vert \partial _t ^n y(x,t) \vert &\le& C_0 \frac{ (2n)!}{R_2^{2n} (2n+1)^\mu},\label{TT3}\\
\vert  \partial _x \partial _t ^n y(x,t) \vert &\le& C_0 \frac{ (2n+1)!}{R_1R_2^{2n} (2n+2)^\mu} \label{TT4}
\end{eqnarray}
for $(x,t)\in [-1,1]\times [t_1,t_2]$ and $n\in \N$. 

We already know that $y\in C^\infty ([-1,1]\times [t_1,t_2] ) $ and that $(y,\partial _x y)\in C ( [-1,1] ,X_{s_0} )$
for some $s_0\in (0,1)$, i.e.  $(y,\partial _x y)\in C( [-1,1],{\mathcal X}_{L_0} )$ with 
$L_0=L(s_0)=e^{2(1-s_0)}L_1\le e^2L_1< (e/2)^2$. 
Thus we have for some constant $C>0$ and for all $n\in \N$ and all $(x,t)\in [-1,1]\times [t_1,t_2]$, 
\begin{eqnarray*}
\vert \partial _t ^{n+1} y(x,t) \vert &\le & C\, L_0^{| n -\frac{1}{2}|+1} \Gamma (n+ \frac{1}{2})^2 (1+n)^{-2},
 \\
\vert \partial _x \partial _t ^{n+1} y(x,t) \vert  &\le & C\,   L_0^{n+1} (n!)^2 (1+n)^{-2}.
\end{eqnarray*}  
Using the estimate $\Gamma (n+\frac{1}{2})\sim \Gamma (n+1)(n+\frac{1}{2})^{-\frac{1}{2}}$ and the estimate
$(n!)^2\sim \sqrt {\pi n } (2n)! / 2^{2n}$ that follows from Stirling formula, we infer the existence of a universal 
constant $C_0>0$ such that  \eqref{TT3}-\eqref{TT4} hold for some constants $R_1,R_2$ with 
$4/e <R_1<R_2<\sqrt{4/L_0}$.  

Assume now that \eqref{TT1} is true for all $k\in \{ 0 , 1, ..., 2l+1 \}$ for some $l\in \N $. Let us pick $k\in \{ 2l, 2l+1\}$, and let us
check that \eqref{TT1} is true for $k+2\in \{ 2l+2, 2l+3\}$. Then 
\begin{eqnarray*}
\vert \partial _x ^{k+2} \partial _t ^n y(x,t)\vert 
&=& \vert \partial _x^k \partial _t^n \partial _x^2 y\vert \\
&=& \vert \partial _x^k \partial _t^n (\partial _t y - f(x,y,\partial _x y)\vert \\
&\le& \vert \partial _x^k \partial _t ^{n+1} y \vert + \sum_{p\ge 1} \vert \partial _x^k\partial _t^n (A_{p,0} (x)y^p)\vert + 
\sum_{q\ge 1} \sum_{p\ge 0} \vert \partial _x^k \partial _t^n (A_{p,q} (x) y^p (\partial _x y)^q)\vert \\
&=:& I_1 + I_2 + I_3. 
\end{eqnarray*}  
Then 
\[
I_1 \le C_l \frac{(2n+2+k)!}{R_1^kR_2^{2n+2} (2n+k+3)^\mu } 
=C_l \left( \frac{R_1}{R_2}  \right) ^2  
\frac{(2n+2+k)!}{R_1^{k+2}R_2^{2n} (2n+k+3)^\mu }\cdot
\]
On the other hand, we have as in the proof of Proposition \ref{prop10} that for some positive constant $K=K(\mu )$
\begin{eqnarray*}
\vert \partial _x ^k\partial _t ^n (A_{p,0} (x) y^p )\vert 
&\le&
\frac{\overline{C} C_l^{p}  K^{p} (2n+k)! }{ R_1^k R_2^{2n} (2n+k+1)^\mu b_0^p } ,\\
\vert \partial _x ^k\partial _t ^n (A_{p,q} (x) y^p (\partial _x y) ^q)\vert 
&\le& 
\frac{\overline{C} C_l^{p+q}  K^{p+q} (2n+k+1)! }{ R_1^k R_2^{2n} (2n+k+1)^\mu  b_0^p b_1^q R_1^q}
\quad \forall q\ge 1.
\end{eqnarray*}
This yields 
\[
I_2+I_3\le 
\sum_{p\ge 1} \frac{\overline{C} C_l^p K^p (2n+k)!}{R_1^k R_2^{2n} (2n+k+1)^\mu b_0^p}
+\sum_{q\ge 1}\sum_{p\ge 0} \frac{\overline{C} C_l^{p+q} K^{p+q} (2n+k+1)!}{R_1^k R_2^{2n} (2n + k+1)^\mu
 b_0^p b_1^qR_1^q}
\cdot
\]
The desired estimate
\be
\vert \partial _x ^{k+2}\partial _t ^n y (x,t)\vert \le C_{l+1} \frac{ (2n+k+2) ! }{R_1^{k+2} R_2^{2n} (2n + k+3)^\mu}  
\label{TT8}
\ee
is satisfied provided that 
\begin{eqnarray}
&&\left( \frac{R_1}{R_2}  \right) ^2  C_l
+ \overline{C} R_1^2 \sum_{p\ge 1}  \frac{1}{(2n+k+1)(2n+k+2)} 
\left( \frac{2n+k+3}{2n+k+1} \right) ^\mu 
\left( \frac{C_lK}{b_0}  \right) ^p \nonumber\\
&&\quad+ \overline{C} R_1^2  \sum_{q\ge 1} \sum_{p\ge 0}  \frac{1}{2n+k+2} 
\left( \frac{2n+k+3}{2n+k+1} \right) ^\mu 
\left( \frac{C_lK}{b_0}  \right) ^{p}  \left( \frac{C_lK}{b_1R_1}  \right) ^{q}  
\le  C_{l+1}\cdot \label{TT10}
\end{eqnarray}
We assume that for some number $\delta \in (0,1)$,
\be
\label{YYY1}
C_l\le \delta \cdot \min (\frac{b_0}{K}, \frac{b_1R_1}{K}). 
\ee
We set 
\[
C_{l+1}:=\lambda_l C_l := [ \left( \frac{R_1}{R_2}  \right) ^2    +   \frac{K}{b_0}
  \frac{\overline{C} R_1^2}{(2l+1)(2l+2)} \frac{3^\mu }{1-\delta} + \frac{K}{b_1}
 \frac{\overline{C} R_1}{2l+2}  \frac{3^\mu }{(1-\delta)^2} ] C_l. 
\]
With this choice, \eqref{TT10}  and \eqref{TT8} are satisfied. 
Since $R_1<R_2$, there exist some number $l_0\in \N$ such that $\lambda _l\le 1$ (and hence $C_{l+1}\le C_l$) for
$l\ge l_0$. For \eqref{YYY1} to be satisfied for all $l\ge 0$, it  remains then to choose $C_0$ sufficiently small so that 
\[
\max (C_0, \lambda _0C_0, \lambda _1\lambda _0C_0, ...., \lambda _{l_0-1} \cdots \lambda _0C_0)\le 
 \delta \cdot \min (\frac{b_0}{K}, \frac{b_1R_1}{K})  \cdot
\]
 
\section*{Acknowledgements}
The first author would like to thank Jean-Michel Coron for introducing him to the use of Gevrey functions in control. The first author was supported by ANR project ISDEEC (ANR-16-CE40-0013). The second author was supported by the ANR project Finite4SoS (ANR-15-CE23-0007).


\end{document}